\documentclass[a4paper,11pt]{article}
\usepackage[english]{babel}

\usepackage[utf8]{inputenc}
\usepackage[T1]{fontenc}

\usepackage{lmodern}
\usepackage{amssymb,amsmath,amsfonts}
\usepackage{xspace}
\usepackage{graphicx}
\usepackage{hyperref}
\usepackage{url}
\usepackage{stmaryrd}
\usepackage{mathrsfs}
\usepackage[babel=true]{csquotes}
\usepackage{amsthm}
\usepackage{subfigure}

\theoremstyle{theorem}
\newtheorem{thm}{Theorem}[section]
\newtheorem{lm}[thm]{Lemma}
\newtheorem{prop}[thm]{Proposition}
\newtheorem{cor}[thm]{Corollary}
\newtheorem*{thmintro}{Theorem}
\newtheorem*{propintro}{Proposition}

\theoremstyle{definition}
\newtheorem{dfn}[thm]{Definition}

\def\dpar#1#2{\frac{\partial #1}{\partial #2}}

\def\Hil{\mathcal{H}\xspace}

\def\N{\mathbb{N}\xspace}
\def\Z{\mathbb{Z}\xspace}
\def\R{\mathbb{R}\xspace}
\def\C{\mathbb{C}\xspace}

\def\classe#1#2{\mathcal{C}^{#1}#2}

\def\Barg{\mathcal{B}\xspace}

\hypersetup{pdfborder={0000}, colorlinks=true, linkcolor=blue}

\begin{document}

\title{Singular Bohr-Sommerfeld conditions for 1D Toeplitz operators: elliptic case}

\author{Yohann Le Floch\footnote{Universit\'e de Rennes 1, IRMAR, UMR 6625,
Campus de Beaulieu, b\^atiments 22 et 23, 263 avenue du G\'en\'eral Leclerc,
CS 74205, 35042  RENNES C\'edex, France; email: yohann.lefloch@univ-rennes1.fr}
}

\maketitle

\begin{abstract} In this article, we state the Bohr-Sommerfeld conditions around a global minimum of the principal symbol of a self-adjoint semiclassical Toeplitz operator on a compact connected K\"{a}hler surface, using an argument of normal form which is obtained thanks to Fourier integral operators. These conditions give an asymptotic expansion of the eigenvalues of the operator in a neighbourhood of fixed size of the singularity. We also recover the usual Bohr-Sommerfeld conditions away from the critical point. We end by investigating an example on the two-dimensional torus.
\end{abstract}

\newpage

\section{Introduction}

Let $M$ be a compact, connected K\"{a}hler manifold of complex dimension $1$, with fundamental 2-form $\omega$. Assume $M$ is endowed with a prequantum line bundle $L$. Let $K$ be another holomorphic line bundle and define the quantum Hilbert space $\Hil_{k}$ as the space of holomorphic sections of $L^k \otimes K$, for every positive integer $k$. The operators acting on $\Hil_{k}$ that we consider are Berezin-Toeplitz operators (\cite{BouGui,BorPauUri2,Cha1,MaMa}, and many others). The semiclassical parameter is $k$, and the semiclassical limit is $k \rightarrow +\infty$. Formally, $k$ is the inverse of Planck's constant $\hbar$.

Our aim is to understand the spectrum of a given self-adjoint Toeplitz operator, in the semiclassical limit. In the setting of ($\hbar$-)pseudodifferential operators, the similar study was done by Colin de Verdi\`ere in \cite{CdV5}. In his article \cite{Cha2}, Charles obtained the description of the intersection of the spectrum of a self-adjoint Toeplitz operator with an interval of regular values of its principal symbol, in the semiclassical limit: the eigenvalues are selected by an integrality condition for some geometric quantities (actions) associated to the symbol of the operator (these are the Bohr-Sommerfeld conditions).

In this article, we extend these conditions around a global minimum of the principal symbol; since we work with only one degree of liberty, we expect to have a very precise description of the eigenvalues near the critical point. 

\subsection{Main theorem}
Let $A_k$ be a self-adjoint Toeplitz operator on $M$; its normalized symbol $a_{0} + \hbar a_{1} + \ldots$ is real-valued. Assume that its principal symbol $a_0$ admits a global minimum at $m_0 \in M$, with $a_0(m_0) = 0$. Denote by $\lambda_{k}^{(1)} \leq \lambda_{k}^{(2)} \leq \ldots \leq \lambda_{k}^{(j)} \leq \ldots$ the eigenvalues of $A_{k}$. Our main result is the following theorem.
\begin{thmintro}[Theorem \ref{thm:BSprec}]
There exist $E_0 > 0$, a sequence $g(.,k)$ of functions of $\classe{\infty}{(\R,\R)}$ which admits an asymptotic expansion of the form $g(.,k) = \sum_{\ell \geq 0}k^{-\ell} g_{\ell}$ in the $\classe{\infty}$ topology, and a positive integer $k_{0}$ such that for every integer $N \geq 1$ and for every $E \leq E^0$, there exists a constant $C_{N} > 0$ such that for $k \geq k_{0}$:
\begin{equation*} \left( \lambda_{k}^{(j)} \leq E \ \text{or} \ E_{k}^{(j)} \leq E \right) \Rightarrow \left| \lambda_{k}^{(j)} - E_{k}^{(j)} \right| \leq C_{N} k^{-N} \end{equation*}
where
\begin{equation*} E_k^{(j)} = g\left(k^{-1} \left(j+\frac{1}{2} \right),k \right), \quad j \in \N. \end{equation*} 
\end{thmintro}
This allows to compute asymptotic expansions to all order for eigenvalues of $A_{k}$ lower than $E^0$, except that so far, we do not know who are the $g_{\ell}$, $\ell \geq 0$, or how to compute them. In fact, $g(.,k)$ is constructed as the local inverse of a sequence $f(.,k)$ which also admits an asymptotic expansion $f(.,k) = \sum_{\ell \geq 0} k^{-\ell} f_{\ell}$ in the $\classe{\infty}$ topology, and the first terms $f_{0}, f_{1}$ are related to geometric quantities (actions) associated to $A_{k}$.

\subsection{Link with the usual Bohr-Sommerfeld conditions} \label{sect:actions}

Let $I$ be a set of regular values of the principal symbol $a_{0}$; for every $E$ in $I$, the level set $\Gamma_{E} := a_{0}^{-1}(E)$ is diffeomorphic to $S^1$. Fix an orientation on $\Gamma_{E}$ depending continuously on $E$. Define the \textit{principal action} $c_{0} \in \classe{\infty}(I)$ in such a way that the parallel transport in $L$ along $\Gamma_{E}$ is the multiplication by $\exp(i c_{0}(E))$. Of course, $c_{0}(E)$ is defined up to an integer multiple of $2 \pi$, but we can always choose a determination of $c_{0}$ that is smooth on $I$.

Let $(\delta,\varphi)$ be a \textit{half-form bundle}, that is a line bundle $\delta \rightarrow M$ together with an isomorphism of line bundles $\varphi:\delta^2 \rightarrow \Lambda^{2,0} T^*M$. 
It is known that for any connected compact K\"ahler manifold of complex dimension 1, such a couple exists. Introduce the hermitian holomorphic line bundle $L_{1}$ such that $K = L_{1} \otimes \delta$. For $E$ in $I$, define the \textit{subprincipal form} $\kappa_{E}$ as the 1-form on $\Gamma_E$ such that 
\begin{equation*}  \kappa_{E}(X_{a_{0}}) = - a_{1} \end{equation*}
where $X_{a_{0}}$ stands for the Hamiltonian vector field associated to $a_{0}$. Denote by $j_{E}$ the embedding $\Gamma_{E} \rightarrow M $, and introduce the connection $\nabla^E$ on $j_{E}^*L_{1} \rightarrow \Gamma_{E}$ defined by
\begin{equation*} \nabla^E = \nabla^{j_{E}^*L_{1}} + \frac{1}{i} \kappa_{E}, \end{equation*}
with $\nabla^{j_{E}^*L_{1}}$ the connection induced by the Chern connection of $L_1$ on $j_{E}^*L_{1}$. Define the \textit{subprincipal action} $c_{1} \in \classe{\infty}{(I)}$ like the principal action, replacing $L$ by $L_1$ endowed with this connection.

Finally, define an index $\epsilon$ from the half-form bundle $\delta$ as follows: the map
\begin{equation*} \varphi_{E}:  \delta_{E}^2  \rightarrow  T^*\Gamma_{E} \otimes \C, \quad u  \mapsto  j_{E}^*\varphi(u) \end{equation*}
is an isomorphism of line bundles. The set
\begin{equation*} \left\{ u \in \delta_{E}; \varphi_{E}(u^{\otimes 2}) > 0 \right\} \end{equation*}
has one or two connected components. In the first case, we set $\epsilon_{E} = 1$, and in the second case $\epsilon_{E} = 0$. In fact, $\epsilon_{E}$ is a constant $\epsilon_{E} = \epsilon$ for $E$ in $I$.

If we define carefully $c_{0}$ and $c_{1}$, the following result holds.
\begin{propintro}[Proposition \ref{prop:f0f1}]
Set $I = ]0,E^0[$. Then 
\begin{equation} f_{0} = \frac{1}{2\pi} c_{0}, \quad  f_{1} = \frac{1}{2\pi} c_{1} \end{equation}
on $I$.
\end{propintro}
Thus, we recover the regular Bohr-Sommerfeld conditions away from the minimum.

\subsection{Structure of the article}
The paper is organized as follows: we start by recalling some properties of Toeplitz operators on a compact manifold. Then, we briefly explain how to adapt the theory in the case where the phase space is the whole complex plane. The fourth section is devoted to the construction of Fourier integral operators that we use to construct our microlocal normal form in the following part. In section 6, we state the Bohr-Sommerfeld conditions and some consequences. In the last section, we investigate an example to give some numerical evidence of our results.

\section{Preliminaries and notations}

First, we introduce the notations and conventions that we will adopt through this whole article. They are already written in \cite{Cha2} for instance, but we recall them here for the sake of completeness.

\subsection{Quantum spaces}

Let $M$ be a connected compact K\"{a}hler manifold, with fundamental 2-form $\omega \in \Omega^2(M,\R)$. Assume $M$ is endowed with a prequantum bundle $L \rightarrow M$, that is a Hermitian holomorphic line bundle whose Chern connection $\nabla$ has curvature $\frac{1}{i} \omega$. Let $K \rightarrow M$ be a Hermitian holomorphic line bundle. For every positive integer $k$, define the quantum space $\Hil_{k}$ as:
\begin{equation*} \Hil_{k} = H^0(M,L^k \otimes K) = \left\{ \text{holomorphic sections of } L^k \otimes K \right\}. \end{equation*}
The space $\Hil_{k}$ is a subspace of the space $L^2(M,L^k \otimes K)$ of sections of finite $L^2$-norm, where the scalar product is given by
\begin{equation*} \langle \varphi, \psi \rangle = \int_{M} h_{k}(\varphi,\psi) \mu_{M} \end{equation*}
with $h_{k}$ the hermitian product on $L^k \otimes K$ induced by those of $L$ and $K$, and $\mu_{M}$ the Liouville measure on $M$. Since $M$ is compact, $\Hil_{k}$ is finite dimensional, and is thus given a Hilbert space structure with this scalar product.

\subsection{Geometric notations}

Unless otherwise mentioned, \enquote{smooth} will always mean $\classe{\infty}{}$, and a section of a line bundle will always be assumed to be smooth. The space of sections of a bundle $E \rightarrow M$ will be denoted by $\Gamma(M,E)$. 

Let $L_{P} \rightarrow P$ and $L_{N} \rightarrow N$ be two prequantum bundles over K\"{a}hler manifolds, whose fundamental 2-forms are denoted by $\omega_{P}$ and $\omega_{N}$. Denote by $p_{1}$ and $p_{2}$ the projections of $P \times N$ on each factor, and $L_{P} \boxtimes L_{N} = p_{1}^* L_{P} \otimes p_{2}^* L_{N}$; then, if $P \times N$ is endowed with the symplectic form $p_{1}^*\omega_{P} +p_{2}^*\omega_{N} $, $L_{P} \boxtimes L_{N} \rightarrow P \times N$ is a prequantum bundle.

Let $P^{op}$ be the manifold $P$ endowed with the symplectic form $-\omega_P$ and the (almost) complex structure opposed to the one of $P$, and let $L_{P}^{-1}$ be the inverse (dual) bundle of $L_{P}$ with induced hermitian and holomorphic structure and connection; then $L^{-1}_P \rightarrow P^{op}$ is a prequantum bundle. If $k$ is a positive integer, we can identify the Schwartz kernel of an operator $T:\Gamma(P,L^k_{P}) \rightarrow \Gamma(N,L^k_{N})$ to a section of $L^k_{N} \boxtimes L^{-k}_{P} \rightarrow N \times P^{op}$ \textit{via} the following formula:
\begin{equation*} Ts(x) = \int_{P} T(x,y).s(y)  \mu_{P}(y), \end{equation*}
where $\mu_{P}$ is the Liouville measure on the manifold $P$.

\subsection{Admissible and negligible sequences}

Let $M$ be a compact connected K\"ahler manifold. Let $(s_{k})_{k \geq 1}$ be a sequence such that for each $k$, $s_{k}$ belongs to $\Gamma(M,L^k \otimes K)$. We say that $(s_{k})_{k \geq 1}$ is 
\begin{itemize}
\item \textit{admissible} if for every positive integer $\ell$, for every vector fields $X_{1},\ldots,X_{\ell}$ on $M$, and for every compact set $C \subset M$, there exist a constant $c > 0$ and an integer $N$ such that
\begin{equation*} \forall m \in C \quad \| \nabla_{X_{1}} \ldots \nabla_{X_{\ell}} s_{k}(m) \| \leq c k^{N}, \end{equation*}
\item \textit{negligible} if for every positive integers $\ell$ and $N$, for every vector fields $X_{1},\ldots,X_{\ell}$ on $M$, and for every compact set $C \subset M$, there exists a constant $c > 0$ such that
\begin{equation*} \forall m \in C \quad \| \nabla_{X_{1}} \ldots \nabla_{X_{\ell}} s_{k}(m) \| \leq c k^{-N}. \end{equation*}
\end{itemize}
We say that $(s_{k})_{k \geq 1}$ is \textit{negligible} over an open set $U \subset M$ if the previous estimates hold for every compact subset of $U$. We denote by $O(k^{-\infty})$ any negligible sequence or the set of negligible sequences. The \textit{microsupport} $\text{MS}(s_{k})$ of an admissible sequence $(s_{k})_{k \geq 1}$ is the complement of the set of points of $M$ which admit a neighbourhood where $(s_{k})_{k \geq 1}$ is negligible. Finally, we say that two admissible sequences $(t_{k})_{k \geq 1}$ and $(s_{k})_{k \geq 1}$ are microlocally equal on an open set $U$ if $\text{MS}(t_{k} - s_{k}) \cap U = \emptyset$; the symbol $\sim$ will indicate microlocal equivalence. We can then define, \textit{via} the sequences of their Schwartz kernels, \textit{admissible} and \textit{smoothing} operators, and the microsupport and microlocal equality of operators.

\subsection{Toeplitz operators}

Let $\Pi_{k}$ be the orthogonal projector of $L^2(M,L^k \otimes K)$ onto $\Hil_{k}$. A \textit{Toeplitz operator} is any sequence $(T_{k}: \Hil_{k} \rightarrow \Hil_{k})_{k \geq 1}$ of operators of the form
\begin{equation} T_{k} = \Pi_{k} M_{f(.,k)} + R_{k} \label{eq:defToeplitz}\end{equation}
where $f(.,k)$ is a sequence of $\classe{\infty}{(M)}$ with an asymptotic expansion $f(.,k) = \sum_{\ell \geq 0} k^{-\ell} f_{\ell}$ for the $\classe{\infty}{}$ topology, $M_{f(.,k)}$ is the operator of multiplication by $f(.,k)$ and $R_{k}$ is a smoothing operator. They are the semiclassical analogue of the Toeplitz operators studied by Boutet de Monvel and Guillemin in \cite{BouGui}.

We recall the following essential theorem about Toeplitz operators, which is a consequence of the works of Boutet de Monvel and Guillemin \cite{BouGui} (see also \cite{Bord,Cha1,Gui}).

\begin{thm}
The set $\mathcal{T}$ of Toeplitz operators is a star algebra whose identity is $(\Pi_{k})_{k \geq 1}$. The contravariant symbol map
\begin{equation*} \sigma_{\text{cont}} : \mathcal{T} \rightarrow \classe{\infty}{(M)}[[\hbar]] \end{equation*}
sending $T_{k}$ into the formal series $\sum_{\ell \geq 0} \hbar^{\ell} f_{\ell}$ is well defined, onto, and its kernel is the ideal consisting of $O(k^{-\infty})$ Toeplitz operators. More precisely, for any integer $\ell$,
\begin{equation*} \| T_{k} \| = O(k^{-\ell}) \ \text{if and only if} \ \sigma_{\mathrm{cont}}(T_{k}) = O(\hbar^{\ell}). \end{equation*}
\end{thm}

We will mainly work with the \textit{normalized symbol}
\begin{equation*} \sigma_{\text{norm}} = \left( \text{Id} + \frac{\hbar}{2} \Delta \right) \sigma_{\text{cont}} \end{equation*}
where $\Delta$ is the holomorphic Laplacian acting on $\classe{\infty}{(M)}$; unless otherwise mentioned, when we talk about a subprincipal symbol, this refers to the normalized symbol. This symbol has the good property that, if $T_{k}$ and $S_{k}$ are Toeplitz operators with respective principal symbols $t_{0}$ and $s_{0}$, then
\begin{equation*} \sigma_{\text{norm}}(T_{k}S_{k}) = t_{0}s_{0} + \frac{\hbar}{2i} \left\{ t_{0},s_{0} \right\} + O(\hbar^2). \end{equation*}

Finally, we will need to apply functional calculus to Toeplitz operators.
\begin{prop}[\cite{Cha1}]
Let $T_{k}$ be a self-adjoint Toeplitz operator with symbol $\sum_{\ell \geq 0} \hbar^{\ell} t_{\ell}$ and $g$ be a function of $\classe{\infty}{(\R,\C)}$. Then $g(T_{k})$ is a Toeplitz operator with principal symbol $g(t_{0})$.
\end{prop}

\section{Toeplitz operators on the complex plane}

\subsection{Bargmann spaces}

We consider the K\"{a}hler manifold $\C \simeq \R^2$ with coordinates $(x,\xi)$, standard complex structure and symplectic form $\omega_{0} = d\xi \wedge dx$. Let $L_0 = \R^2 \times \C \rightarrow \R^2$ be the trivial fiber bundle with standard hermitian metric $h_{0}$ and connection $\nabla^0$ with $1$-form $\frac{1}{i} \alpha$, where $\alpha_u(v) = \frac{1}{2} \omega_0(u,v)$; endow $L_{0}$ with the unique holomorphic structure compatible with $h_{0}$ and $\nabla^0$.
For every positive integer $k$, the quantum space that we consider is
\begin{equation*} \Hil_k^0 = H^0(\R^2,L_{0}^k) \cap L^2(\R^2,L_{0}^k); \end{equation*}
this means that in this case, we make the arbitrary choice that the auxiliary line bundle $K$ is the trivial bundle with flat connection.
These spaces coincide with Bargmann spaces \cite{Bar, Bar2}, which are spaces of square integrable functions with respect to a Gaussian weight. More precisely, we choose the holomorphic coordinate $z=\frac{x-i\xi}{\sqrt{2}}$ and note
\begin{equation*} \Barg_{k} = \left\{ f\psi^k; f:\C \mapsto \C \ \text{holomorphic}, \int_{\R^2} |f(z)|^2 \exp(-k|z|^2) \ d\lambda(z) < +\infty \right\} \end{equation*}
with $\psi:\C \rightarrow \C, z \mapsto  \exp\left( -\frac{1}{2} |z|^2 \right)$, $\psi^k$ its $k$-th tensor power, and $\lambda$ the Lebesgue measure on $\R^2$. It is easily shown that for $k \geq 1$, $\Hil_k^0$ is precisely $\Barg_{k}$. Sometimes, we will use the identification of the section $f \psi$ to the function $f$ in abusive notations, such as talking about the operator $\frac{\partial}{\partial z}$ action on $\Barg_k$, etc.
It is standard that $\Barg_{k}$ is closed in $L^2\left(\R^2,\exp(-k|z|^2) d\lambda(z) \right)$, and is thus a Hilbert space; moreover, we know an orthonormal basis of $\Barg_{k}$.
\begin{prop} The family $(\varphi_{n,k})_{n \in \N}$, where $\varphi_{n,k}(z) = \sqrt{\frac{k^{n+1}}{2\pi n!}} \ z^n  \psi^k$, is an orthonormal basis of $\Barg_{k}$. \end{prop}
We denote by $\Pi_{k}^0$ the orthogonal projector from $L^2(\R^2,L_{0}^k)$ onto $\Barg_{k}$.

\subsection{Admissible and negligible sequences}

Since we will only deal with $\classe{\infty}$ sections, we can adopt the same definitions for admissible and negligible sequences as in the previous section. 

\subsection{Toeplitz operators}
\label{subsection:ToepBarg}

To consider Toeplitz operators acting on Bargmann spaces without raising technical issues, we could only work with operators with compactly supported kernels. However, we would miss the simple case of the harmonic oscillator. So we need to introduce symbol classes, very similar to the ones used when dealing with $\hbar$-pseudodifferential operators (see for instance \cite{DS}). The proofs of the results of this part are collected in the appendix.

Let $d$ be a positive integer. For $u$ in $\C^{d}$, set $m(u) = \left(1 + \|u\|^2 \right)^{\frac{1}{2}}$. For every integer $j$, we define the symbol class $\mathcal{S}_{j}^d$ as the set of sequences of functions of $\classe{\infty}{(\C^{d})}$ which admit an asymptotic expansion of the form $a(.,k) = \sum_{\ell \geq 0} k^{-\ell}a_{\ell}$ in the sense that
\begin{itemize}
\item $\forall \ell \in \N \quad  \forall \alpha, \beta \in \N^{2d} \quad \exists \ C_{\ell,\alpha,\beta} > 0 \quad |\partial_{z}^{\alpha} \partial_{\bar{z}}^{\beta} a_{\ell}| \leq C_{\ell,\alpha,\beta} m^j$,
\item $\forall L \in \N^* \quad \forall \alpha, \beta \in \N^{2d} \quad \exists  \ C_{L,\alpha} > 0 \quad \left| \partial_{z}^{\alpha} \partial_{\bar{z}}^{\beta} \left(a - \sum_{\ell=0}^{L-1} k^{-\ell} a_{\ell} \right)  \right| \leq C_{L,\alpha,\beta} k^{-L} m^j$.
\end{itemize}
Set $\mathcal{S}^d = \bigcup_{j \in \Z} \mathcal{S}_{j}^d$. Now, let $a(.,k)$ be a symbol in $\mathcal{S}_{j}^1$, and consider the operator 
\begin{equation} A_k = \text{Op}(a(.,k)) = \Pi_{k}^0 M_{a(.,k)} \Pi_{k}^0 \label{eq:ToepC}\end{equation}
acting on the subspace
\begin{equation*} \mathfrak{S}_k = \left\{ \varphi \in \Barg_k; \ \forall j \in \N \quad \sup_{z \in \C} \left( |\varphi(z)|(1 + |z|^2)^{j/2}  \right) < + \infty  \right\} \end{equation*}	
of $\Barg_k$. As shown in \cite{Bar2}, $\mathfrak{S}_k$ corresponds to the Schwartz space \textit{via} a particular unitary mapping between $L^2(\R)$ and $\Barg_k$, the Bargmann transform. It is easily seen that $A_k$ sends $\mathfrak{S}_k$ into $\mathfrak{S}_k$; it is even continuous $\mathfrak{S}_k \rightarrow \mathfrak{S}_k$. Note that if $j=0$, then $A_k$ is bounded $\Barg_k \rightarrow \Barg_k$, and its norm is lower than $\sup |a(.,k)|$.

Let $t$ be the section of $L_{0} \rightarrow \R^2$ with constant value 1. Let $F_{0}$ be the section of $L_0 \boxtimes L^{-1}_{0}$ given by
\begin{equation*} F_{0}(z_1,z_2) = \exp \left(-\frac{1}{2} \left(|z_1|^2 + |z_2|^2 - 2 z_1 \bar{z}_2 \right) \right) t(z_1) \otimes  t^{-1}(z_2), \end{equation*}
or equivalently, if $u = (x,\xi)$ where $z = \frac{1}{\sqrt{2}}(x - i\xi)$,
\begin{equation*} F_{0}(u,v) = \exp \left(-\frac{1}{4}\|u-v\|^2 - \frac{i}{2} \omega_{0}(u,v) \right) t(u) \otimes  t^{-1}(v). \end{equation*}
Adapting the result of section $1.c$ of \cite{Bar}, with the good normalization for the weight defining our Bargmann spaces, we have the following:
\begin{prop} $\Pi_{k}^0$ admits a Schwartz kernel given by $\frac{k}{2 \pi} F_{0}^k $. \end{prop}
In the rest of the paper, we will use the same letter to designate an operator and its Schwartz kernel. This proposition allows us to compute the Schwartz kernel of any Toeplitz operator.
\begin{lm}Let $a(.,k)$ be a symbol in $\mathcal{S}_{j}^1$; then $A_k = \text{Op}(a(.,k))$ admits a Schwartz kernel given by
\begin{equation} \begin{split} A_k(z_1,z_2) = \frac{k}{2 \pi} \exp\left( -\frac{k}{2} \left( |z_1|^2 + |z_2|^2 - 2 z_1 \bar{z}_2 \right) \right) \tilde{a}(z_1,z_2,k)  \\ + R_k \exp\left(-Ck|z_1 - z_2|^2\right), \end{split} \label{eq:SchToep} \end{equation}
where $\tilde{a}(.,.,k)$ belongs to $\mathcal{S}_{j}^2$, $R_k$ is negligible and $C$ is some positive constant. Moreover, one has 
\begin{equation} \tilde{a}(z,z,k) = \left(\exp\left(k^{-1}\Delta \right)a\right)\left(z,k\right)  \label{eq:covcont}\end{equation}
where $\Delta = \frac{\partial}{\partial_z} \frac{\partial}{\partial_{ \bar{z}}} $ is the holomorphic Laplacian acting on $\classe{\infty}{(\C^2)}$, in the sense that the asymptotic expansion of $\tilde{a}(.,.,k)$ is obtained by applying the formal asymptotic expansion of the operator $\exp\left(k^{-1}\Delta \right)$ to the asymptotic expansion of $a(.,k)$.
\label{lm:kernelToep}\end{lm}
This leads us to the following definition.
\begin{dfn}
A \textit{Toeplitz operator} is an operator from $\mathfrak{S}_k$ to $\mathfrak{S}_k$ of the form
\begin{equation} \Pi_k^0 M_{a(.,k)} \Pi_k^0 + S_k,\end{equation}
where $a(.,k)$ is a symbol in $\mathcal{S}^1$ and the kernel of $S_k$ satisfies
\begin{equation} S_k(z_1,z_2) = R_k(z_1,z_2) \exp\left(-Ck|z_1 - z_2|^2\right)  \label{eq:negexp}\end{equation}
with $R_k$ negligible and $C$ some positive constant. As in the compact case, $ \sigma_{\text{cont}}(A_k) = \sum_{\ell \geq 0} \hbar^{\ell} a_{\ell}$ is called the \textit{contravariant symbol} of $A_k$. We denote by $\mathcal{T}_j$ the set of Toeplitz operators with contravariant symbol belonging to $\mathcal{S}_j^1$.
\end{dfn}
In fact, lemma \ref{lm:kernelToep} defines the \textit{covariant symbol} of $A_{k}$:
\begin{equation*} \sigma_{\text{cov}}(A_k)(z) = \sum_{\ell \geq 0} \hbar^{\ell} \tilde{a}_{\ell}(z,z).\end{equation*}
The following lemma gives an important property of the latter.
\begin{lm}
If the covariant symbol of $A_{k}$ vanishes, then the Schwartz kernel of $A_{k}$ is of the form (\ref{eq:negexp}).
\label{lm:covBarg}\end{lm}
As a corollary of lemmas \ref{lm:kernelToep} and \ref{lm:covBarg}, we obtain the stability under composition of the set of Toeplitz operators. 
\begin{cor}
Let $A_k \in \mathcal{T}_j$ and $B_k \in \mathcal{T}_{j'}$ be two Toeplitz operators. Then $C_k = A_k B_k$ belongs to $\mathcal{T}_{j+j'}$; more precisely, its contravariant symbol is given by
\begin{equation} \sigma_{\text{cont}}(C_k)(z) = \left(\exp\left( -\hbar \frac{\partial}{\partial z_1} \frac{\partial}{\partial \bar{z}_2} \right) \sigma_{\text{cont}}(A_{k})(z_1) \sigma_{\text{cont}}(B_{k})(z_2) \right)_{|z_1=z_2=z}, \label{eq:prodcont}\end{equation}
in the same sense as in lemma \ref{lm:kernelToep}.
\label{cor:compoToep}\end{cor}
Define the \textit{normalized symbol} as in the compact case:
\begin{equation*} \sigma_{\text{norm}} = \left( Id + \frac{\hbar}{2} \Delta \right) \sigma_{\text{cont}}. \end{equation*}
From formula (\ref{eq:prodcont}), we find that $\sigma_{\text{norm}}(A_{k} B_{k}) =a_{0} b_{0} + \frac{\hbar}{2 i} \left\{ a_0,b_0 \right\} + O(\hbar^2)$, as expected.
\begin{dfn} A Toeplitz operator $A_k \in \mathcal{T}_j$ is said to be \textit{elliptic at infinity} if there exists some $c > 0$ such that for $z$ in $\C$, $\left| \sigma_{\text{cont}}(A_k)(z) \right| \geq c (1 + |z|^2)^\frac{j}{2}$. \end{dfn}
Adapting proposition $12$ of \cite{Cha1} and theorem $39$ of \cite{CdV2}, one can show the following:
\begin{prop}[Functional calculus] If $A_k$ belongs to $\mathcal{T}_j$ for some $j \geq 1$, is essentially self-adjoint and elliptic at infinity and if $\eta:\R \rightarrow \R$ is a compactly supported $\classe{\infty}$ function, then $\eta(A_k)$ belongs to $\mathcal{T}_{j'}$ for every $j' < 0$. \label{prop:funcBarg}\end{prop}

\section{Fourier integral operators}

The aim of this section is to construct microlocally unitary operators between $\Hil_{k}$ and $\Barg_{k}$, given a local symplectomorphism $\chi$ from $M$ to $\R^2$. In \cite{BouGui}, Boutet de Monvel and Guillemin introduced Fourier integral operators in the homogeneous Toeplitz setting. In the semiclassical Toeplitz theory, such operators between compact manifolds have been used by Charles \cite{Cha4,Cha5}, but some difficulties arise when dealing with a non compact manifold. Nevertheless, the ideas, based on Lagrangian sections, are very similar.

Let $\chi : \Omega_{1} \subset M \rightarrow \Omega_{2} \subset \R^2$ be a symplectomorphism between the open sets $\Omega_{1}$ and $\Omega_{2}$. Then the graph 
\begin{equation*} \Lambda_{\chi} = \left\{ (u,\chi(u)); u \in \Omega_{1} \right\} \subset \Omega_{1} \times \Omega_{2} \end{equation*}
of $\chi$ is a Lagrangian submanifold of the product $M \times \C^{op}$. As in the previous section, let $t$ be the section of $L_{0} \rightarrow \R^2$ with constant value 1. By definition of the connection on $L_{0}$, we have $\nabla^0 t = \frac{1}{i} \alpha \otimes t$, where $\alpha$ is the primitive of $\omega_{0}$ given by $\alpha_{u}(v) = \frac{1}{2} \omega_{0}(u,v)$. The following lemma is elementary.
\begin{lm} Taking $\Omega_{1}$ smaller if necessary, we can find a local gauge $s$ of $L \rightarrow \Omega_{1}$ such that $\nabla s = \frac{1}{i} \chi^* \alpha \otimes s$.  \end{lm}
We consider the section $t_{\Lambda_{\chi}}$ of $L \boxtimes L^{-1}_{0}$ over $\Lambda_{\chi}$ given by
\begin{equation*} t_{\Lambda_{\chi}}(u,\chi(u)) = s(u) \otimes t^{-1}(\chi(u)).  \end{equation*}
Thanks to proposition $2.1$ of \cite{Cha4}, we can build a local section $E$ of $L \boxtimes L^{-1}_{0} \rightarrow \Omega_{1} \times \Omega_{2}^{op}$ such that
\begin{itemize}
\item $E$ is equal to $t_{\Lambda_{\chi}}$ on $\Lambda_{\chi}$,
\item for every holomorphic vector field $Z$ on $\Omega_{1} \times \Omega_{2}^{op}$, the covariant derivative of $E$ with respect to $\bar{Z}$ is zero modulo a section vanishing to infinite order along $\Lambda_{\chi}$,
\end{itemize}
and this section is unique modulo a section vanishing to infinite order along $\Lambda_{\chi}$.
Furthermore, we can always assume that $\|E\| < 1$ outside $\Lambda_{\chi}$, and we will make this assumption until the end of this article. 

We consider a sequence of functions of $\classe{\infty}{(\Omega_{1} \times \Omega_{2}^{op})}$, which admits an asymptotic expansion $\sum_{\ell \geq 0} k^{-\ell} a_{\ell}$ for the $\classe{\infty}{}$ topology whose coefficients are all supported in a fixed (independent of $k$) compact set $C \subset \Omega_{1} \times \Omega_{2}$.
Let $S_{k}$ be the local section of $(L^k \otimes K) \boxtimes L_{0}^{-k} $ given by
\begin{equation*} S_{k}(u,v) = \frac{k}{2 \pi} E^k(u,v) a(u,v,k),  \end{equation*}
and consider the operator $S_{k}$ defined on $ \Gamma(\C,L_{0}^k)$ by
\begin{equation*} (S_{k}\phi)(u) = \int_{\R^2} S_{k}(u,v).\phi(v) \ d\lambda(v),\end{equation*}
which makes sense since $S_{k}(.,.)$ vanishes outside $\Omega_{1} \times \Omega_{2}$.
\begin{prop} The operator $R_{k} = S_{k} \Pi_{k}^0$ maps $ \Barg_{k}$ into $\Gamma(M,L^k \otimes K)$. \end{prop}
\begin{proof} We must show that if $\varphi_{k}$ is a smooth square integrable section of $L_{0}^k  \rightarrow \C$, then $S_{k} \Pi_{k}^0  \varphi_{k}$ is a smooth section of $L^k \otimes K \rightarrow M$. It is enough to show that the Schwartz kernel of $S_{k} \Pi_{k}^0$ and its derivatives with respect to the first variable are rapidly decreasing in the second variable.
Let $R_{k}$ be this kernel; one has 
\begin{equation*} R_{k}(u,v) = \int_{p_{2}(C)} S_{k}(u,w).\Pi_{k}^0(w,v) dw, \end{equation*}
with $p_{2}$ the projection from $M \times \C$ to $\C$. So
\begin{equation*}  R_{k}(u,v) = \left( \frac{k}{2 \pi} \right)^2 \int_{p_{2}(C)} f(u,v,w,k)  E^k(u,w).t^k(w)\otimes t^{-k}(v) dw
\end{equation*}
with $f(u,v,w,k) = a(u,w,k) \exp \left(-\frac{k}{4}\|w-v\|^2 - \frac{ik}{2} \omega_{0}(w,v) \right)$.
This implies the estimate
\begin{equation*} \|R_{k}(u,v)\| \leq \left( \frac{k}{2 \pi} \right)^2 \int_{p_{2}(C)} | a(u,w,k) |   \exp \left(-\frac{k}{4}\|w-v\|^2 \right) \| E^k(u,w) \| dw. \end{equation*}
Since $\|E\| \leq 1$ and $a(.,.,k)$ is bounded by some constant $c_{k} > 0$, this yields
\begin{equation*} \|R_{k}(u,v)\| \leq \left( \frac{k}{2 \pi} \right)^2 c_{k} \int_{p_{2}(C)} \exp \left(-\frac{k}{4}\|w-v\|^2 \right) dw.  \end{equation*}
Using $\|w-v\|^2  \geq  \| v \|^2 - 2  \| w \| \| v \|$, we obtain
\begin{equation*} \|R_{k}(u,v)\| \leq \left( \frac{k}{2 \pi} \right)^2 c_{k} \exp \left(-\frac{k}{4}  \| v \|^2 \right) \int_{p_{2}(C)} \exp \left(\frac{k}{2} \| w \| \| v \| \right) dw.  \end{equation*}
Finally, an upper bound for the integral that appears in this inequality is $\pi r^2 \exp \left( \frac{k}{2} r \| v \| \right)$ where $p_{2}(C)$ is included in the closed ball of radius $r$ and centered at the origin. This allows us to conclude that
\begin{equation*} \|R_{k}(u,v)\| \leq \left( \frac{k}{2 \pi} \right)^2 c_{k} \pi r^2  \exp \left( -\frac{k}{4}  \| v \|^2 + \frac{kr}{2}  \| v \| \right). \end{equation*}

The same kind of estimates hold for the successive derivatives of $R_{k}$ with respect to $u$; we prove them by differentiating under the integral sign.
\end{proof}
Unfortunately, $R_{k}$ has no reason to map holomorphic sections to holomorphic sections; to fix this problem, we set $T_{k} = \Pi_{k} R_{k}$; defined in this way, $T_{k}$ is an operator from $\Barg_{k}$ to $\Hil_{k}$.
\begin{prop} The Schwartz kernel of $T_{k}$ reads
\begin{equation} T_{k}(u,v) = \frac{k}{2 \pi} E^k(u,v) b(u,v,k) + O(k^{-\infty}) \label{eq:FIO}\end{equation}
with $b(.,.,k)$ a sequence of smooth functions which admits an asymptotic expansion $b(.,.,k) = \sum_{\ell \geq 0} k^{-\ell} b_{\ell}$ for the $\classe{\infty}{}$ topology satisfying
\begin{equation*} b_{0}(u,\chi(u),k) = \mu(u) a_{0}(u,\chi(u),k) \end{equation*}
where $\mu$ is a smooth, nowhere vanishing function which depends only on the section $E$.
\end{prop}
The proof is the same as the proof of proposition $4.2$ of \cite{Cha4}; it is based on an application of the stationary phase lemma. 

An operator $V_{k}:\Barg_{k} \rightarrow \Hil_{k}$ admitting a Schwartz kernel of the form of equation (\ref{eq:FIO}) and satisfying $\Pi_{k} V_{k} \Pi_{k}^0 = V_{k} $ will be called a \textit{Fourier integral operator} associated to the sequence $b(.,.,k)$; let us denote by $\mathcal{FI}(\chi)$ the set of such operators. We define the \textit{full symbol} map
\begin{equation*} \sigma: \mathcal{FI}(\chi)  \rightarrow  \classe{\infty}{(M)}[[\hbar]], \quad V_{k}  \mapsto  \sum_{\ell \geq 0} \hbar^{\ell} b_{\ell}(u,\chi(u)). \end{equation*} 
One can show that its kernel consists of smoothing operators. In the same way, we define $\mathcal{FI}(\chi^{-1}) : \Hil_{k} \rightarrow \Barg_{k}$. The following property is another application of the stationary phase lemma.
\begin{prop}
Let $R_{k} \in \mathcal{FI}(\chi)$ and $S_{k} \in \mathcal{FI}(\chi^{-1})$ with respective principal symbols $r_{0}(u,\chi(u))$ and $s_{0}(v,\chi^{-1}(v))$. Then there exists a smooth nowhere vanishing function $\nu: \R^2 \rightarrow \R$ such that for every Toeplitz operator $T_{k}$ on $M$ with principal symbol $t_{0}$, $S_{k} T_{k} R_{k}$ is a Toeplitz operator on $\R^2$ with principal symbol equal to
\begin{equation*}  \nu(v)  s_{0}(v,\chi^{-1}(v)) t_{0}(\chi^{-1}(v)) r_{0}(\chi^{-1}(v),v) \end{equation*}
on $\Omega_{2}$.
\end{prop}

To conclude this section, we prove that we can find \textit{microlocally unitary} operators mapping $\Barg_{k}$ to $\Hil_{k}$ in the following sense. 
\begin{prop} There exists a Fourier integral operator $U_{k}: \Barg_{k} \rightarrow \Hil_{k}$ such that
\begin{itemize}
\item $U_{k}^*U_{k} \sim \Pi_{k}^0$ on $\Omega_{2}$,
\item $U_{k}U_{k}^* \sim \Pi_{k}$ on $\Omega_{1}$,
\end{itemize}
reducing $\Omega_{1}$ and $\Omega_{2}$ if necessary, where we recall that the symbol $\sim$ stands for microlocal equality.
\label{prop:MU}\end{prop}
\begin{proof} 
Start from a Fourier integral operator $S_{k}$ associated to $s(.,.,k)$ with principal symbol $s_{0}(v,\chi^{-1}(v))$ never vanishing on $\Omega_{1} \times \Omega_{2}$.
The first step is to construct an operator with the first property; we do it by induction, correcting $S_{k}$ by Toeplitz operators. More precisely, let $P_k$ be a Toeplitz operator on $\R^2$, and denote by $p_0$ its principal symbol. Set $U_k^{(0)} = S_k P_k$; then $U_k^{(0)*} U_k^{(0)}$ is a Toeplitz operator on $\R^2$, with principal symbol $\nu(v) \left|p_0(v)\right|^2 \left|s_0\left(\chi^{-1}(v),v \right) \right|^2$. Since $\nu(v)$ and $s_0(s,\chi^{-1}(v))$ vanish in no point $v$, one can choose $p_0$ such that this principal symbol is equal to $1$. Doing so, $U_k^{(0)*} U_k^{(0)}$ has the same principal symbol as $\Pi_{k}^0$, so there exists a Toeplitz operator $R_k^{(0)}$ such that
\begin{equation*} U_k^{(0)*} U_k^{(0)} \sim \Pi_{k}^0 + k^{-1} R_k^{(0)} \end{equation*}
on $\Omega_{2}$. From now on, when there is no ambiguity, the equality between operators will mean microlocal equality on $\Omega_{2}$.
Let $n \in \N$ and assume that there exists an operator $U_k^{(n)}: \Barg_{k} \rightarrow \Hil_{k}$ and a Toeplitz operator $R_k^{(n)}$ (with principal symbol $r_{n}$) such that
\begin{equation*} U_k^{(n)*} U_k^{(n)} = \Pi_{k}^0 + k^{-(n+1)} R_k^{(n)}. \end{equation*}
Let $T_{k}$ be a Toeplitz operator on $\R^2$ with principal symbol $t_{0}$, and set $U_k^{(n+1)} = U_{k}^{(n)} \left( \Pi_{k}^0 + k^{-(n+1)}T_{k} \right)$. One has
\begin{equation*} U_k^{(n+1)*} U_k^{(n+1)} =  \Pi_{k}^0 + k^{-(n+1)} \left( T_{k}^* + R_k^{(n)} + T_{k} \right) + k^{-(n+2)} R_k^{(n+1)}  \end{equation*}
with $R_k^{(n+1)}$ a Toeplitz operator. This implies that if we choose $t_{0}$ such that $ 2 \Re(t_{0}) = - r_{n}$, then
\begin{equation*} U_k^{(n+1)*} U_k^{(n+1)} =  \Pi_{k}^0 +  k^{-(n+2)} R_k^{(n+1)}. \end{equation*}
So we can construct the operators $U_{k}^{(n)}$ by induction; it remains to apply Borel's summation lemma to find the desired operator $U_{k}$.

Composing such a $U_{k}$ by $U_{k}$ on the left and $U_{k}^*$ on the right gives:
\begin{equation*} \left( U_{k} U_{k}^* \right)^2 \sim U_{k} U_{k}^* \quad \text{on \  } \Omega_{1}. \end{equation*}
Since $U_{k} U_{k}^*$ is an elliptic Toeplitz operator on $\Omega_{1}$ (its principal symbol vanishes nowhere), it has a microlocal inverse at each point of $\Omega_{1}$; so the preceding equation yields $U_{k}U_{k}^* \sim \Pi_{k}$ on $\Omega_{1}$.
\end{proof}
\paragraph{}
Of course, such operators satisfy the analogue of Egorov's theorem:
\begin{prop} If $U_{k}$ is as above, then, for every Toeplitz operator $T_{k}$ on $M$ with principal symbol $t_{0}$, $S_{k} = U_{k}^* T_{k} U_{k}$ is a Toeplitz operator on $\R^2$ with principal symbol equal to $t_{0} \circ \chi^{-1}$ on $\Omega_{2}$. \end{prop}
For a proof, see \cite[theorem 47]{CdV2}.
The action of a Fourier integral operator at the subprincipal level is much more complicated to compute. Denote by $\sigma_{0}(U_{k})$ the principal symbol of $U_{k}$, and by $\gamma$ the 1-form on $\Lambda_{\chi}$ such that
\begin{equation*} \nabla^{\text{Hom}(\C,K)} \sigma_{0}(U_{k}) = -\frac{1}{i} \gamma \otimes \sigma_{0}(U_{k}) \end{equation*}
endowing $\C$ with the trivial connection and $K$ with the one inherited from $L_{1}$ and $\delta$. Now, notice that the symplectomorphism $\chi$ brings the complex structure of $\C^{op}$ to a positive complex structure $j$ on $M$. Introduce the section $ \Psi$ of $\mathrm{Hom}(\Omega^{1,0}(\C),\Omega_{j}^{1,0}(M))_{|\Lambda_{\chi}} \rightarrow \Lambda_{\chi}$ such that for all $\alpha \in \Lambda^{1,0}(T\Omega_{2}^{op})^*$ and $\beta \in \Lambda^{1,0}(T\Omega_{1})^*$,
\begin{equation*} \Psi(\alpha) \wedge \bar{\beta}  = (\chi^* \alpha) \wedge \bar{\beta}.\end{equation*}
This map is well-defined because the sesquilinear pairing $ \Lambda_{j}^{1,0}(T_{m}\Omega_{1})^* \times \Lambda^{1,0}(T_{m}\Omega_{1})^*  \rightarrow \C, (\alpha,\beta) \mapsto (\alpha \wedge \bar{\beta})/\omega_{m}$ is non-degenerate. 
Let $\delta$ be the 1-form on $\Lambda_{\chi}$ such that 
\begin{equation*} \nabla^{} \Psi = \delta \otimes \Psi \end{equation*}
where $\nabla^{}$ is the connection induced by the Chern connections of $\Omega^{2,0}(\C)$ and $\Omega^{2,0}(M)$.
In \cite[Theorem $3.3$]{Cha5}, Charles derived the following formula.
\begin{thm} With the same notations as in the previous proposition and denoting by $t_{1}$ the subprincipal symbol of $T_{k}$, the subprincipal symbol $s_{1}$ of $S_{k}$ is given by:
\begin{equation*}  s_{1}(u) = t_{1}(m)  + \left\langle {\gamma}_{(m,u)} - \frac{1}{2} {\delta}_{(m,u)} , \left(X_{t_{0}}(m),((\chi^{-1})^*X_{t_{0}})(u) \right) \right\rangle  \end{equation*}
for $u \in \R^2$ and $m=\chi^{-1}(u)$.
\label{thm:FIOssp}\end{thm}

\section{Microlocal normal form}

\subsection{The local model}

Our local model will be the realization of the quantum harmonic oscillator in the Bargmann representation: $Q_{k} = \frac{1}{k} \left( z \dpar{}{z} + \frac{1}{2} \right)$, with domain $\C[z]$ which is dense is $\Barg_{k}$. The following lemma is easily shown.
\begin{lm} $Q_{k}$ is an essentially self-adjoint Toeplitz operator with normalized symbol $q_{0}$. Moreover, the spectrum of $Q_k$ is
\begin{equation*} \text{Sp}(Q_{k}) = \left\{ k^{-1} \left(n+\frac{1}{2}\right), n \in \N \right\}.\end{equation*} \end{lm}

\subsection{A symplectic Morse lemma}

Let $(N,\omega)$ be a two-dimensional symplectic manifold and $f$ a function of $\classe{\infty}{(N,\R)}$. Assume $f$ admits an elliptic critical point at $n_{0} \in N$, with $f(n_{0})=0$. Replacing $f$ by $-f$ if necessary, we can assume that this critical point is a local minimum for $f$. Define
\begin{equation*} q_{0}: \R \rightarrow \R^2, \quad (x,\xi) \mapsto \frac{1}{2}(x^2 + \xi^2) . \end{equation*}
The following theorem is well-known.
\begin{thm}
There exist a local symplectomorphism $\chi:(N,n_{0})  \rightarrow (\R^2,0)$ and a function $g$ in $\classe{\infty}{(\R,\R )}$ satisfying $g(0) = 0$ and $g'(0) > 0$, such that
\begin{equation*} f \circ \chi^{-1} = g \circ q_{0} \end{equation*}
where $\chi^{-1}$ is defined.
\label{thm:symp}\end{thm}
It can be viewed as a consequence of the isochore Morse lemma \cite{CdV1} or of Eliasson's symplectic normal form theorem \cite{Elia}, but this case is in fact easier than these two results.

\subsection{Semiclassical normal form}

We consider a self-adjoint Toeplitz operator $A_{k}$ on $M$; its normalized symbol 
\begin{equation*} a(.,\hbar) = a_{0} + \hbar a_{1} + \ldots \end{equation*}
is real-valued.
Assume that the principal symbol $a_{0}$ admits a non-degenerate local minimum at $m_{0} \in M$. Assume also that $a_{0}(m_{0}) = 0$, so that $a_{0}$ takes positive values on a neighbourhood of $m_{0}$.
Hence, thanks to theorem \ref{thm:symp}, we get a neighbourhood $\Omega_{1}$ of $m_{0}$ in $M$, a neighbourhood $\Omega_{2}$ of $0$ in $\R^2$, a local symplectomorphism $\chi:\Omega_{1} \rightarrow \Omega_{2}$ and a function $g_{0}$ of $\classe{\infty}{(\R,\R)}$ with $g_{0}(0) = 0$ and $g_{0}'(0) > 0$, such that:
\begin{equation*}  a_{0} \circ \chi^{-1} = g_{0} \circ q_{0} \end{equation*}
on $\Omega_{2}$. We denote by $f_{0}$ the local inverse of $g_{0}$. Our goal is to show:
\begin{thm}
There exist a Fourier integral operator $U_{k}: \Barg_{k} \rightarrow \Hil_{k}$ and a sequence $f(.,k)$ of functions of $\classe{\infty}{(\R,\R)}$ which admits an asymptotic expansion in the $\classe{\infty}{}$ topology of the form $f(.,k) = \sum_{\ell \geq 0} k^{-\ell}f_{\ell}$, such that:
\begin{itemize}
\item $U_{k}^*U_{k} \sim \Pi_{k}^0$ on $\Omega_{2}$,
\item $U_{k}U_{k}^* \sim \Pi_{k}$ on $\Omega_{1}$,
\item $U_{k}^*f(A_{k},k)U_{k} \sim Q_{k}$ on $\Omega_{2}$.
\end{itemize}
\label{thm:formenormale}
\end{thm}

\begin{proof}
We consider an operator $U_{k}^{(0)}$ satisfying the two first points (see the previous section). We will construct the operator that we seek by successive perturbations by unitary Toeplitz operators on $\Barg_{k}$. More precisely, we show by induction that for every positive integer $n$, there exist an operator $U_{k}^{(n)}: \Barg_{k} \rightarrow \Hil_{k}$ satisfying the two first points, a sequence $f^{(n)}(.,k)$ of functions of $\classe{\infty}{(\R,\R)}$ of the form $f^{(n)}(.,k) = \sum_{\ell=0}^n k^{-\ell}f_{\ell}$, with $f_{\ell}$ smooth, and a Toeplitz operator $R_{k}^{(n)}$ acting on $\Barg_{k}$ such that
\begin{equation*} U_{k}^{(n)*} f^{(n)}(A_k,k) U_{k}^{(n)} = Q_k + k^{-(n+1)} R_{k}^{(n)} \quad \text{on \ } \Omega_{2}. \end{equation*}

The first step is as follows: by the results of the previous section, the operator $U_{k}^{(0)*}f_{0}(A_{k})U_{k}^{(0)}$ is a Toeplitz operator on $\Barg_{k}$, whose principal symbol is equal to $ f_{0} \circ a_{0} \circ \chi^{-1} = q_{0}$ on $\Omega_{2}$. Hence, there exists a Toeplitz operator $R_k^{(0)}$ on $\Barg_k$ such that
\begin{equation*} U_{k}^{(0)*}f_{0}(A_{k})U_{k}^{(0)} = Q_k + k^{-1} R_k^{(0)}. \end{equation*}  
We look for $U_{k}^{(1)}$ of the form $U_k^{(0)} P_k$ with $P_k$ a unitary Toeplitz operator on $\Barg_k$.
Moreover, we choose $f^{(1)}(.,k) = f_0 + k^{-1} \theta_1 \circ f_0$ with $\theta_1$ a smooth function that remains to determine.
Expanding, we get
\begin{equation*} U_{k}^{(1)*} f^{(1)}(A_k,k) U_{k}^{(1)} = P_k^* U_{k}^{(0)*} f_{0}(A_{k}) U_{k}^{(0)} P_k +  k^{-1}  P_k^* U_{k}^{(0)*} (\theta_{1} \circ f_{0})(A_{k})  U_{k}^{(0)} P_k \end{equation*}
which yields
\begin{equation*} U_{k}^{(1)*} f^{(1)}(A_k,k) U_{k}^{(1)} = P_k^* \left( Q_k + k^{-1} R_k^{(0)} \right) P_{k} + k^{-1}  P_k^* U_{k}^{(0)*} (\theta_{1} \circ f_{0})(A_{k})  U_{k}^{(0)} P_k. \end{equation*}
Consequently, we wish to have
\begin{equation*} P_k^* \left( Q_k + k^{-1} R_k^{(0)} \right) P_{k} + k^{-1}  P_k^* U_{k}^{(0)*} (\theta_{1} \circ f_{0})(A_{k})  U_{k}^{(0)} P_k = Q_{k} + k^{-2} R_{k}^{(1)}  \end{equation*}
where $R_{k}^{(1)}$ is a Toeplitz operator; this amounts, remembering that $P_{k}$ is unitary, to
\begin{equation*} Q_k P_{k} + k^{-1} \left( R_k^{(0)} P_{k} + U_{k}^{(0)*} (\theta_{1} \circ f_{0})(A_{k})  U_{k}^{(0)} P_k \right) = P_{k} Q_{k} + k^{-2} P_{k} R_{k}^{(1)}  \end{equation*}
which we can rewrite
\begin{equation*} [Q_{k},P_{k}] + k^{-1} \left( R_k^{(0)} P_{k} + U_{k}^{(0)*} (\theta_{1} \circ f_{0})(A_{k})  U_{k}^{(0)} P_k \right) = k^{-2} P_{k} R_{k}^{(1)} . \end{equation*}
This will be true if and only if the subprincipal symbol of the operator on the left of the equality vanishes, that is to say 
\begin{equation*} \frac{1}{i} \left\{ q_{0},p_{0} \right\} + p_{0} (r_{0}  + \theta_{1} \circ q_{0} ) = 0  \end{equation*}
where $p_{0}$ and $r_{0}$ stand for the respective principal symbols of $P_{k}$ and $R_{k}^{(0)}$.
Set $p_{0} = \exp(i \varphi_{0})$ with $\varphi_{0}$ a smooth, real-valued function (since $P_{k}$ is unitary). The previous equation then becomes
\begin{equation*} \left\{ \varphi_{0},q_{0} \right\} = r_{0}  + \theta_{1} \circ q_{0}.  \end{equation*}
This equation is standard and it is well-known that it can be solved. We recall a method from \cite{Elia} to find $\theta_{1}$ and $\varphi_{0}$ smooth such that it is satisfied, since we will need to know how to construct these in part \ref{subsection:computation}. 
Consider the functions:
\begin{equation*} F(x,\xi) = -\frac{1}{2\pi} \int_{0}^{2\pi} r_{0}(\phi_{q_{0}}^t(x,\xi)) \ dt \end{equation*}
and
\begin{equation*} \varphi_{0}(x,\xi) = -\frac{1}{2\pi} \int_{0}^{2\pi} t \  r_{0}(\phi_{q_{0}}^t(x,\xi)) \ dt, \end{equation*}
where $\phi_{q_{0}}^t$ stands for the Hamiltonian flow of $q_{0}$ taken at time $t$:
\begin{equation*} \phi_{q_{0}}^t(x,\xi) = (x \cos t + \xi \sin t, -x \sin t + \xi \cos t). \end{equation*}
Since we integrate on a compact set and the flow $\phi^t$ is smooth with respect to $(x,\xi)$, both $F$ and $\varphi_{0}$ are smooth. By construction, we have $\left\{ F,q_{0} \right\} = 0$. But we have the easy lemma
\begin{lm}
Let $f$ be a function of $\classe{\infty}{(\R^2,\R)}$ such that $\left\{ f,q_{0} \right\} = 0$. Then the function $g$ such that
\begin{equation*} f = g \circ q_{0} \end{equation*}
belongs to $\classe{\infty}{(\R,\R)}$.
\label{lm:q0}\end{lm}
So there exists a function $\theta_{1}$ of $\classe{\infty}{(\R,\R)}$ such that $F = \theta_{1} \circ q_{0}$.
Integrating by parts, it is easy to show that
\begin{equation*}  \left\{\varphi_{0},q_{0}\right\} = \theta_{1} \circ q_{0} + r_{0} \end{equation*}

The next steps are practically the same; indeed, let $n \geq 1$ and assume that we have found $U_{k}^{(n)}$ and $f^{(n)}(.,k)$ satisfying the desired properties. We now look for $U_{k}^{(n+1)}$ of the form $U_k^{(n)} (\Pi_{k}^0 + k^{-n} V_k)$ with $V_k$ a Toeplitz operator on $\Barg_k$ such that $\Pi_{k}^0 + k^{-n} V_k$ is unitary. Furthermore, we write $f^{(n+1)}(.,k) = f^{(n)}(.,k) + k^{-(n+1)} \theta_{n+1} \circ f_0$ with $\theta_{n+1}$ an unknown smooth function. We want the existence of a Toeplitz operator $R_{k}^{(n+1)}$ such that
\begin{equation*} U_k^{(n)*} f^{(n+1)}(A_{k},k) U_k^{(n)}  \left( \Pi_{k}^0 + k^{-n} V_k \right) =  ( \Pi_{k}^0 + k^{-n} V_k) \left(Q_{k} + k^{-(n+2)} R_{k}^{(n+1)} \right)  \end{equation*}
which gives, expanding,
\begin{equation*} \begin{split}  Q_{k} &+ k^{-n} Q_{k} V_{k} + k^{-(n+1)} \left( R_{k}^{(n)} +  U_k^{(n)*} (\theta_{n+1} \circ f_0)(A_{k}) U_k^{(n)} \right) = Q_{k} +  k^{-n} V_{k} Q_{k} \\
& +  k^{-(n+2)} R_{k}^{(n+1)} +  k^{-(2n+1)} V_{k} R_{k}^{(n+1)}.
\end{split}
\end{equation*}
Thus, we wish that
\begin{equation*} [Q_{k},V_{k}] + k^{-1} \left( R_{k}^{(n)} +  U_k^{(n)*} (\theta_{n+1} \circ f_0)(A_{k}) U_k^{(n)} \right) = 0;  \end{equation*}
this will be verified if and only if
\begin{equation*} \frac{1}{i} \left\{ q_{0},v_{0} \right\} + r_{n} + \theta_{n+1} \circ q_{0} = 0 \end{equation*}
which is treated as before.

We conclude thanks to Borel's summation lemma (applied to both $f(.,k)$ and $U_{k}$).
\end{proof}

\section{Bohr-Sommerfeld conditions}

Let $A_{k}$ be a self-adjoint Toeplitz operator on $M$ as in the previous section. Moreover, assume that $a_{0}(m_{0})$ is a global minimum of the principal symbol $a_{0}$, and that $m_{0}$ is the unique point of $M$ with this property. This implies that there exists $E^0 > 0$ such that for every $E \leq E^0$, the level set $a_{0}^{-1}(E)$ is connected and contained in $\Omega_{1}$.

The maximum norm $\| A_{k} \|_{\infty}$ of $A_{k}$ tends to the $L^{\infty}$-norm $\| a_{0} \|_{\infty}$ of $a_{0}$ as $k$ goes to infinity \cite{Bord}; hence, for $k$ large enough, the spectrum of $A_{k}$ is included in the set $\left [ -E^1,E^1 \right]$, where $E^{1} = \| a_{0} \|_{\infty} + 1$.

\subsection{Statement of the result}

Before stating the Bohr-Sommerfeld conditions, it is convenient to show that the sequence $f(.,k)$ can be inverted, and that its inverse still has a good asymptotic expansion.
\begin{lm}
For $k$ large enough, the function $f(.,k)$ that appears in theorem \ref{thm:BSprec} is a bijection from $\left[ -E^1,E^0 \right]$ to its image; more precisely, it is strictly increasing. Moreover, the inverse sequence $g(.,k)$ admits an asymptotic expansion in the $\classe{\infty}$ topology of the form $g(.,k) = \sum_{\ell \geq 0} k^{-\ell} g_{\ell} + O(k^{-\infty})$, uniformly on $\left[ -E^1,E^0 \right]$.
\label{lm:inversemicro}\end{lm}
\begin{proof}
The first assertion follows from the mean value inequality
\begin{equation*}\forall k \geq 1 \quad \forall E,\tilde{E} \in \left[ -E^1,E^0 \right] \qquad |f(E,k) - f(\tilde{E},k)| \geq \inf_{\left[ -E^1,E^0 \right]} |f'(.,k)| |E-\tilde{E}|\end{equation*}
and the fact that $f'(.,k)$ is bounded below by some positive constant. This implies that for $k$ sufficiently large, $f(.,k)$ is strictly monotone on $\left[ -E^1, E^0 \right]$; since $f_{0}'(0) > 0$, $f(.,k)$ is in fact strictly increasing. For the second part, the proof is once again based on Borel's summation lemma; it is done by induction thanks to Taylor's formula with integral remainder.
\end{proof}
We can therefore introduce the sequences 
\begin{equation} E_{k}^{(j)} = g\left(k^{-1} \left(j+\frac{1}{2} \right),k \right), \quad j \in \N \label{eq:defEk}\end{equation}
for $k$ large enough and for $j$ such that $k^{-1}\left(j+\frac{1}{2}\right)$ belongs to the set $\left[ f(-E^1,k),f(E^0,k) \right]$. Since $g(.,k)$ is also strictly increasing, the $E_{k}^{(j)}$ are ordered:
\begin{equation*} \forall j \in \N, \quad E_{k}^{(j)} < E_{k}^{(j+1)}. \end{equation*}
We can be more precise; fix $j \in \N$ and write 
\begin{equation*} E_{k}^{(j)} = g_{0}\left( k^{-1}\left( j + \frac{1}{2} \right) \right) + k^{-1} g_{1}\left( k^{-1}\left( j + \frac{1}{2} \right) \right) + O(k^{-2}). \end{equation*}
Then, applying Taylor's formula with integral remainder, we get
\begin{equation} E_{k}^{(j)} = \underbrace{g_{0}(0)}_{=0} + k^{-1} \left( g_{1}(0) + \left( j + \frac{1}{2} \right)g_{0}'(0) \right) + O(k^{-2}). \label{eq:DAEk}\end{equation}
One must be careful with this estimate: the $O(k^{-2})$ remainder is no longer uniform with respect to $j$.
Denote by $\lambda_{k}^{(1)} \leq \lambda_{k}^{(2)} \leq \ldots \leq \lambda_{k}^{(j)} \leq \ldots$ the eigenvalues of $A_{k}$.
The main result of this section is the following theorem.
\begin{thm}
There exists a positive integer $k_{0} \geq 1$ such that for every integer $N \geq 1$ and for every $E \leq E^0$, there exist a constant $C_{N} > 0$ such that for $k \geq k_{0}$:
\begin{equation} \left( \lambda_{k}^{(j)} \leq E \ \text{or} \ E_{k}^{(j)} \leq E \right) \Rightarrow \left| \lambda_{k}^{(j)} - E_{k}^{(j)} \right| \leq C_{N} k^{-N}. \label{eq:DA}\end{equation} 
Moreover, for $k$ large enough, all the eigenvalues of $A_{k}$ smaller than $E^0$ are simple. In particular, we obtain an asymptotic expansion to every order for the eigenvalues of $A_{k}$ smaller than $E^0$.
\label{thm:BSprec}\end{thm}

We will need the following lemma, based on the min-max principle.

\begin{lm}[{\cite[lemma $3.3$]{SanLau}}]
Let $A$ and $B$ be two self-adjoint operators acting, respectively, on the Hilbert spaces $\Hil'$ and $\Hil$, both bounded from below. Denote by $\Pi_{I}^A$ the spectral projection of $A$ on $I$ and by $\lambda_{1}^A \leq \lambda_{2}^A \leq \ldots \leq \lambda_{j}^A \leq \ldots$ the increasing sequence of eigenvalues below the essential spectrum of $A$; if there is a finite number $j_{\text{max}}$ of such eigenvalues, extend the sequence for $j > j_{\text{max}}$ by setting $\lambda_{j}^A = \lambda_{\text{ess}}^A$, where $\lambda_{\text{ess}}^A$ is the infimum of the essential spectrum of $A$. Introduce the same notations for $B$. Suppose that there exist a bounded operator $U:\Hil \rightarrow \Hil'$, an interval $I = (-\infty,E]$, and constants $C > 0$, $c \in (0,1)$ such that $U \Pi_{I}^B(\Hil) \subset \text{Dom}(A)$ and 
\begin{equation*} \left\| (U^* A U - B) \Pi_{I}^B \right\| \leq C \end{equation*} 
and
\begin{equation*} \left\| U^*U\Pi_{I}^B -  \Pi_{I}^B \right\| \leq c. \end{equation*}
Then, for all $j$ such that $\lambda_{j}^B \leq E$, one has 
\begin{equation*} \lambda_{j}^A \leq (\lambda_{j}^B + C) \left( 1 + \frac{c}{1-c} \right). \end{equation*}
\label{lm:BNF}\end{lm}

\begin{proof}[Proof of theorem \ref{thm:BSprec}]

Fix $E$ in $\left[-E^{1},E^0\right]$. Let $\mathcal{J}$ be an open neighbourhood of $\left[-E^{1},E\right]$ such that the open set $a_0^{-1}(\mathcal{J}) $ is contained in $\Omega_{1}$, and let $\eta: \R \rightarrow \R$ be a smooth function equal to $1$ on $\left[-E^{1},E\right]$ and $0$ outside $\mathcal{J}$.  Consider the Toeplitz operator $R_{k} = \eta(A_{k})$ and set $B_{k} = \left( f(A_{k},k) - U_{k} Q_{k} U_{k}^{*} \right) R_{k}$. By the choice of $R_{k}$, the microsupport of $B_{k}$ is a subset of $\Omega_{1}$. Moreover, $f(A_{k},k)$ is microlocally equal to $U_{k} Q_{k} U_{k}^{*}$ on $\Omega_{1}$. These two facts imply that $B_{k}$ is negligible; since $M$ is compact, this yields that for every $N \geq 1$, there exists a positive constant $C_{N}$ such that
\begin{equation*} \| B_{k} \| \leq C_{N} k^{-N}. \end{equation*}
Now, let $\Pi_{\leq f(E,k)}^{f(A_{k},k)}$ be the spectral projection associated to $f(A_{k},k)$ and corresponding to the eigenvalues smaller than $f(E,k)$. If $(\lambda,\varphi)$ is an eigencouple for $A_{k}$ with $\lambda \leq E$, then $R_{k}\varphi = \eta(\lambda) \varphi = \varphi$. This implies that for every $\phi$ in $\Pi_{\leq f(E,k)}^{f(A_{k},k)}(\Hil_{k})$, $R_{k} \phi = \phi$, and consequently
\begin{equation*} \left\| \left( f(A_{k},k) - U_{k} Q_{k} U_{k}^{*} \right) \Pi_{\leq f(E,k)}^{f(A_{k},k)} \right\| \leq C_{N} k^{-N}. \end{equation*}
Similarly, there exists $c_{N} > 0$ such that
\begin{equation*} \left\| \left( \Pi_{k} - U_{k} U_{k}^{*} \right) \Pi_{\leq f(E,k)}^{f(A_{k},k)} \right\| \leq c_{N} k^{-N}. \end{equation*}
So lemma \ref{lm:BNF} shows that if $f(\lambda_{k}^{(j)},k) \leq f(E,k)$, the inequality
\begin{equation*} k^{-1}\left( j + \frac{1}{2} \right) \leq \left( 1 + \frac{c_{N}k^{-N}}{1 - c_{N}k^{-N}}\right) \left( f(\lambda_{k}^{(j)},k) + C_{N} k^{-N} \right)  \end{equation*}
holds. So for $k$ large enough (independently of $E$), we have $ k^{-1}\left( j + \frac{1}{2} \right) \leq f(E,k)$.

Now, let $\rho$ be a smooth function equal to $1$ on $\left[f_{0}(-E^{1}),f_{0}(E)\right]$ and vanishing outside an open neighbourhood $\mathcal{K}$ of   $\left[f_{0}(-E^{1}),f_{0}(E)\right]$ such that $q_{0}^{-1}(\mathcal{K}) \subset \Omega_{2}$. Thanks to proposition \ref{prop:funcBarg}, we can consider the Toeplitz operator $S_{k} = \rho(Q_{k})$, and set $C_{k} = \left( U_{k}^{*} f(A_{k},k) U_{k} -  Q_{k}  \right) S_{k}$. Since $S_k$ belongs to every $\mathcal{T}_j$, $j < 0$, $C_k$ belongs to $\mathcal{T}_0$ and is thus a bounded operator $\Barg_k \rightarrow \Barg_k$. Moreover, by construction, it is negligible. Hence there exists a positive constant $\tilde{C}_{N}$ such that 
\begin{equation*} \left\| C_{k} \right\| \leq \tilde{C}_{N} k^{-N};  \end{equation*}
modifying $C_{N}$ if necessary, we can assume that $\tilde{C}_{N}$ is equal to $C_{N}$. So, introducing the spectral projection $\Pi_{\leq f(E,k)}^{Q_{k}}$ corresponding to the eigenvalues of $Q_{k}$ smaller than $f(E,k)$, the inequality 
\begin{equation*} \left\| \left( U_{k}^{*} f(A_{k},k) U_{k} -  Q_{k}  \right)  \Pi_{\leq f(E,k)}^{Q_{k}} \right\| \leq C_{N} k^{-N} \end{equation*}
holds. Similarly, we have
\begin{equation*} \left\| \left( U_{k}^{*}U_{k} -  \Pi_{k}^0  \right) \Pi_{\leq f(E,k)}^{Q_{k}} \right\| \leq c_{N} k^{-N}. \end{equation*}
Hence, applying again lemma \ref{lm:BNF}, we obtain that
\begin{equation*} f(\lambda_{k}^{(j)},k) \leq \left( 1 + \frac{c_{N}k^{-N}}{1 - c_{N}k^{-N}}\right) \left(  k^{-1}\left( j + \frac{1}{2} \right) + C_{N} k^{-N} \right)  \end{equation*}
as soon as $k^{-1}\left( j + \frac{1}{2} \right) \leq f(E,k)$. This shows that if $f(\lambda_{k}^{(j)},k) \leq f(E,k)$, then 
\begin{equation*} \left| f(\lambda_{k}^{(j)},k) - k^{-1}\left( j + \frac{1}{2} \right) \right| \leq C'_{N} k^{-N} \end{equation*}
for some positive constant $C'_{N}$. Exchanging the roles of $Q_{k}$ and $A_{k}$, and using lemma \ref{lm:inversemicro}, this gives formula (\ref{eq:DA}).

Using this result and the fact that there exists $c > 0$ such that for $j \in \N$, $E_{k}^{(j+1)} - E_{k}^{(j)}$ is equivalent to $ck^{-1}$, we obtain that the $\lambda_{k}^{(j)}$ are simple for $k$ large enough. 
\end{proof}

\subsection{Computation of the principal and subprincipal terms}
\label{subsection:computation}

In order to exploit these results, it remains to compute a few first terms in the asymptotic expansion of the sequence $f(.,k)$. What we can do is relate the principal and subprincipal terms to the actions introduced in section \ref{sect:actions}.
\begin{prop}
Set $I = ]0,E^0[$. Then 
\begin{equation} f_{0} = \frac{1}{2\pi} c_{0}, \quad  f_{1} = \frac{1}{2\pi} c_{1} \end{equation}
on $I$.
\label{prop:f0f1}\end{prop}

\begin{proof}

Let us first compute $f_{0}$. Fix a level $E$ in $I$. Let $\frac{1}{i} \beta$ be the 1-form describing locally the Chern connection on $L$; then $c_{0}(E)$ is given by
\begin{equation*} c_{0}(E) = \int_{\Gamma_{E}} \beta. \end{equation*}
Using the relation $a_{0} \circ \chi^{-1} = g_{0} \circ q_{0}$, we can then write
\begin{equation*} c_{0}(E) = \int_{\mathcal{C}_{E}} (\chi^{-1})^* \beta \end{equation*}
where $\mathcal{C}_{E}$ is the circle centered at the origin and with radius $\sqrt{2 f_{0}(E)}$. Using Stokes' formula and the fact that $\chi$ is a symplectomorphism, this yields that $c_{0}(E)$ is the area of the disk bounded by $\mathcal{C}_{E}$, that is, if the orientation that we chose is the one giving the positive area (and this is what we will assume in the rest of this section)
\begin{equation} c_{0}(E) = 2\pi f_{0}(E). \label{eq:principal}\end{equation}

\paragraph{}

Now, turn back to the proof of our normal form theorem \ref{thm:formenormale}, where $f_{1}$ is constructed from the subprincipal symbol $r_{0}$ of $U_{k}^{(0)*}f(A_{k},k)U_{k}^{(0)}$. By uniqueness of $f_{1}$, instead of starting from any operator $U_{k}^{(0)}$, we can choose one with symbol $u \otimes v$, where $u$ is constant and $v$ is a square root of $\Psi$. Doing so, we can compute $r_{0}$ thanks to theorem \ref{thm:FIOssp}: 
\begin{equation*} r_{0} = (a_{1} \circ \chi^{-1}) (\ f_{0}' \circ a_{0}   \circ \chi^{-1}) - \nu_{\chi^{-1}(.)}\left(X_{f_{0} \circ a_{0}} \circ \chi^{-1} \right)  \end{equation*}
where $\nu$ is the local connection 1-form associated to $\nabla^{L_1}$. We have $f_{1} = \theta_{1} \circ f_{0}$ with $\theta_{1}$ such that for all $(x,\xi)$ in $\R^2$
\begin{equation*} \left( \theta_{1} \circ q_{0} \right)(x,\xi) =  -\frac{1}{2\pi} \int_{0}^{2\pi} r_{0} \left( \phi_{q_{0}}^t(x,\xi) \right) dt  \end{equation*} 
where $\phi_{q_{0}}^t$ stands for the Hamiltonian flow of $q_{0}$. Since $q_{0} = f_{0} \circ a_{0} \circ \chi^{-1} $, this implies that for $(x,\xi)$ in $\R^2$
\begin{equation*} \begin{split} \left( f_{1} \circ a_{0} \circ \chi^{-1} \right)(x,\xi) =  -\frac{1}{2\pi} \int_{0}^{2\pi} (a_{1} \ f_{0}' \circ a_{0} ) \left(  \chi^{-1}\left(\phi_{q_{0}}^t(x,\xi)\right) \right) dt \\ 
+ \frac{1}{2\pi} \int_{0}^{2\pi} \nu_{\chi^{-1}(\phi_{q_{0}}^t(x,\xi))} \left(X_{f_{0} \circ a_{0}} \left(\chi^{-1}(\phi_{q_{0}}^t(x,\xi))\right) \right) dt. \end{split} \end{equation*}
So, for $m \neq m_{0}$ in $\Omega_{1}$, we have
\begin{equation*} \begin{split} (f_{1} \circ a_{0})(m) = - \frac{1}{2\pi} \int_{0}^{2\pi} (a_{1} \ f_{0}' \circ a_{0} ) \left( \chi^{-1}\left(\phi_{q_{0}}^t(\chi(m))\right) \right) dt \\
 + \frac{1}{2\pi} \int_{0}^{2\pi} \nu_{\chi^{-1}(\phi_{q_{0}}^t(\chi(m)))} \left(X_{f_{0} \circ a_{0}} \left(\chi^{-1}\left(\phi_{q_{0}}^t(\chi(m))\right)\right) \right) dt; \end{split} \end{equation*}
thus, if $E = a_{0}(m)$,
\begin{equation*} \begin{split} f_{1}(E) =  \frac{1}{2\pi} \int_{0}^{2\pi} \nu_{\chi^{-1}(\phi_{q_{0}}^t(\chi(m)))} \left(X_{f_{0} \circ a_{0}} \left(\chi^{-1}\left(\phi_{q_{0}}^t(\chi(m))\right)\right) \right) dt \\
 - \frac{f_{0}'(E)}{2\pi} \int_{0}^{2\pi} a_{1} \left( \chi^{-1}\left(\phi_{q_{0}}^t(\chi(m))\right) \right) dt.  \end{split} \end{equation*}
But $\chi^{-1} \circ \phi_{q_{0}}^t \circ \chi$ is the Hamiltonian flow of $q_{0} \circ \chi = f_{0} \circ a_{0}$, so 
\begin{equation*} \begin{split} f_{1}(E) = \frac{1}{2\pi} \int_{0}^{2\pi} \nu_{\phi_{f_{0} \circ a_{0}}^t(m)} \left(X_{f_{0} \circ a_{0}}(\phi_{f_{0} \circ a_{0}}^t(m)) \right)dt  \\  -\frac{f_{0}'(E)}{2\pi} \int_{0}^{2\pi} a_{1} \left( \phi_{f_{0} \circ a_{0}}^t(m) \right) dt. \end{split} \end{equation*}
Therefore, if $T_{E}$ is the period of the flow $\phi_{a_{0}}^t$ along $\Gamma_{E}$, we have $T_{E} = 2\pi f_{0}'(E)$ for $E$ close to $0$, and a change of variable gives
\begin{equation*} \begin{split} f_{1}(E) = \frac{1}{2\pi f_{0}'(E)} \int_{0}^{2\pi} \nu_{\phi_{f_{0} \circ a_{0}}^{tf_{0}'(E)}(m)} \left(X_{f_{0} \circ a_{0}} \left( \phi_{f_{0} \circ a_{0}}^{t f_{0}'(E)}(m) \right) \right)dt \\
 - \frac{1}{2\pi} \int_{0}^{T_{E}} a_{1} \left( \phi_{f_{0} \circ a_{0}}^{tf_{0}'(E)}(m) \right) dt \end{split} \end{equation*}
which yields, since the Hamiltonian vector field associated to $f_{0} \circ a_{0}$ is $X_{f_{0} \circ a_{0}} = (f_{0}' \circ a_{0}) X_{a_{0}}$, and hence $ \phi_{f_{0} \circ a_{0}}^{t f_{0}'(E)}(m) = \phi_{a_{0}}^{t}(m)$:
\begin{equation*} \begin{split} f_{1}(E) = \frac{1}{2\pi f_{0}'(E)} \int_{0}^{2\pi} \nu_{\phi_{a_{0}}^t(m)} \left( f_{0}'(E) X_{a_{0}}(\phi_{a_{0}}^t(m)) \right)dt \\
- \frac{1}{2\pi} \underbrace{\int_{0}^{T_{E}} a_{1} \left( \phi_{a_{0}}^{t}(m) \right) dt}_{= -\int_{\Gamma_{E}} \kappa_{E} } \end{split} \end{equation*}
and by linearity of $\nu$
\begin{equation*}  f_{1}(E) = \frac{1}{2\pi} \underbrace{\int_{0}^{2\pi} \nu_{\phi_{a_{0}}^t(m)} \left( X_{a_{0}}(\phi_{a_{0}}^t(m)) \right)dt}_{= \int_{\Gamma_{E}} \nu } + \frac{1}{2\pi} \int_{\Gamma_{E}} \kappa_{E}.  \end{equation*}
The right term of this equality is precisely equal to $c_{1}(E)$; so we have on $I$
\begin{equation} c_{1} = 2 \pi f_{1}. \end{equation}

\end{proof}

To conclude, we can show that $\epsilon = 1$ on $I$, so the result of theorem \ref{thm:BSprec} matches the usual Bohr-Sommerfeld conditions on the set $I$ of regular values.

\subsection{First terms of the asymptotic expansion of the eigenvalues}

Theorem \ref{thm:BSprec} and formula (\ref{eq:DAEk}) give an asymptotic expansion for the eigenvalues of $A_{k}$ smaller than $E^0$. Fix $j \in \N$; for $k$ large enough, one has
\begin{equation*} \lambda_{k}^{(j)} = k^{-1} \left( g_{1}(0) + \left( j + \frac{1}{2} \right)g_{0}'(0) \right) + O(k^{-2}). \end{equation*}
We can be more precise, since we know the value of $g_{1}(0)$: by definition of $g(.,k)$, we have $ g_{1} = - (f_{1} \circ g_{0}) g'_{0}$, and the computation made in the previous part leads to $ f_{1}(0) = - a_{1}(0) f_{0}'(0)$; consequently, $ g_{1}(0) = a_{1}(0)$ and
\begin{equation} \lambda_{k}^{(j)} = k^{-1} \left( a_{1}(0) + \left( j + \frac{1}{2} \right)g_{0}'(0) \right) + O(k^{-2}). \end{equation}
But $g_{0}'(0) = \frac{1}{f_{0}'(0)}$; moreover, it is standard that the principal action $c_{0}$ is smooth even at the critical value $E=0$. Hence, thanks to formula (\ref{eq:principal}), one has
\begin{equation} \lambda_{k}^{(j)} = k^{-1} \left( a_{1}(0) +  \frac{2\pi \left( j + \frac{1}{2} \right)}{c_{0}'(0)} \right) + O(k^{-2}). \label{eq:firstev} \end{equation}
In particular, the gap between two consecutive eigenvalues is given by
\begin{equation} \lambda_{k}^{(j+1)} - \lambda_{k}^{(j)} = \frac{2\pi k^{-1}}{c_{0}'(0)}  + O(k^{-2}). \label{eq:gap}\end{equation}

\section{An example on the torus} 

The aim of this section is to give numerical evidence for our results by investigating the case of a particular Toeplitz operator on the torus $\mathbb{T}$ of real dimension $2$. One can find the details of the quantization of $\mathbb{T}$ in \cite{CM}, where the authors investigate some conjectures on knot states; let us briefly recall its main ingredients.

\subsection{The setting}
Endow $\R^2$ with the linear symplectic form $\omega_{0}$ and consider a lattice $\Lambda$ with symplectic volume $4\pi$. The Heisenberg group $H = \R^2 \times U(1)$ with product
\begin{equation*}(x,u).(y,v) = \left(x+y,uv\exp\left(\frac{i}{2}\omega_{0}(x,y)\right) \right) \end{equation*}
acts on the trivial bundle $L_0 \rightarrow \R^2$, with action given by the same formula. This action preserves the prequantum data, and the lattice $\Lambda$ injects into $H$; therefore, the fiber bundle $L_{0}$ reduces to a prequantum bundle $L$ over $\mathbb{T} = \R^2 \slash \Lambda$. The action extends to the fiber bundle $L_{0}^k$ by
\begin{equation*} (x,u).(y,v) = \left(x+y,u^k v\exp\left(\frac{ik}{2}\omega_{0}(x,y)\right) \right) \end{equation*}
and naturally induces an action
\begin{equation*} T^*:  \Lambda  \rightarrow  \text{End}\left(\Gamma \left(\R^2,L_{0}^k \right)\right), \quad u  \mapsto  T_{u}^*. \end{equation*}
The Hilbert space $\Hil_{k} = H^0(M,L^k)$ can naturally be identified to the space $\Hil_{\Lambda,k}$ of holomorphic sections of $L_{0}^k \rightarrow \R^2$ which are invariant under the action of $\Lambda$, endowed with the hermitian product
\begin{equation*} \langle \varphi,\psi \rangle = \int_{D} \varphi \bar{\psi} \ |\omega_{0}| \end{equation*}
where $D$ is the fundamental domain of the lattice. Furthermore, $\Lambda \slash 2k$ acts on $\Hil_{\Lambda,k}$. Let $e$ and $f$ be generators of $\Lambda$ satisfying $\omega_{0}(e,f) = 4 \pi$; one can show that there exists an orthonormal basis $(\psi_{\ell})_{\ell \in \Z \slash 2k\Z}$ of $\Hil_{\Lambda,k}$ such that
\begin{equation} \forall \ell \in \Z \slash 2k\Z \qquad \left\{\begin{array}{c} T^*_{e/2k} \psi_{\ell} = w^{\ell} \psi_{\ell}  \\ T^*_{f/2k} \psi_{\ell} = \psi_{\ell + 1} \end{array}\right. \end{equation}
with $w = \exp\left( \frac{i \pi}{k} \right)$. The sections $\psi_{\ell}$ can be expressed in terms of $\Theta$ functions.

Set $M_{k} = T^*_{e/2k}$ and $L_{k} = T^*_{f/2k}$. Let $(q,p)$ be coordinates on $\R^2$ associated to the basis $(e,f)$ and $\left[q,p\right]$ be the equivalence class of $(q,p)$. Both $M_{k}$ and $L_{k}$ are Toeplitz operators, with respective principal symbols $ \left[q,p\right]  \mapsto  \exp(2 i \pi p)$ and $ \left[q,p\right] \mapsto \exp(2 i \pi q) $, and vanishing subprincipal symbols. Consequently
\begin{equation*} A_{k} = M_{k} + M_{k}^{-1} + L_{k} + L_{k}^{-1} \end{equation*}
is a Toeplitz operator on $\mathbb{T}$ with principal symbol 
\begin{equation} a_{0}(q,p) = 2 \left( \cos(2\pi p) + \cos(2\pi q) \right) \end{equation}
and vanishing subprincipal symbol. Its matrix in the basis $(\psi_{\ell})_{\ell \in \Z \slash 2k\Z}$ is
\begin{equation*}  \begin{pmatrix} 2 \alpha_0 & 1 & 0 & \ldots & 0 & 1 \\ 1 & \ddots & \ddots & \ddots &  & 0\\ 0 & \ddots & \ddots & \ddots & \ddots & \vdots \\ \vdots & \ddots & \ddots & \ddots & \ddots & 0 \\ 0 &  & \ddots & \ddots  & \ddots & 1 \\ 1 & 0 &  \ldots & 0 & 1 & 2 \alpha_{2k-1}  \end{pmatrix}  \end{equation*}
where
\begin{equation*} \alpha_{\ell} = \cos\left( \frac{\ell \pi}{k} \right).  \end{equation*}
The principal symbol $a_{0}$ is known as Harper's Hamiltonian. It admits a global minimum at $m_{0} = \left[1/2, 1/2 \right]$, with $a_{0}\left(m_{0} \right) = -4$. Figure \ref{fig:contour} is a contour plot of this function on the fundamental domain. In figure \ref{fig:spectre}, we derived numerically the spectrum of $A_{k}$. Figure \ref{fig:module} shows the modulus of the eigenfunction associated to the eigenvalue closest to a prescribed level $E$.

In this situation, we can express $c_{0}'(0)$, and so, by equation (\ref{eq:firstev}), the first eigenvalues of $A_{k}$. Indeed, the symplectic form on $\mathbb{T}$ is $\omega = 4\pi dp \wedge dq$ and the hessian of $a_{0}$ at $m_{0} = \left[1/2,1/2 \right]$ is given by:
\begin{equation*} d^2a_{0}(m_{0}) =  8 \pi^2 I_{2}  \end{equation*}
so it is easy to obtain
\begin{equation*} c'_{0}(0) = 1. \end{equation*}
Consequently, the first eigenvalues are given by:
\begin{equation} \lambda_{k}^{(j)} = - 4 + 2 \pi k^{-1} \left( j + \frac{1}{2} \right) + O(k^{-2}). \label{eq:DAtore}\end{equation}
This is exactly what our simulations show; we plotted the eigenvalues located in a window of length $20 \pi k^{-1}$ around the minimum (so we expect to see about ten eigenvalues) and the result can be seen in figure \ref{fig:zoom}.

\subsection{Figures}

\begin{figure}[hp]
\begin{center}
\includegraphics[scale=0.4]{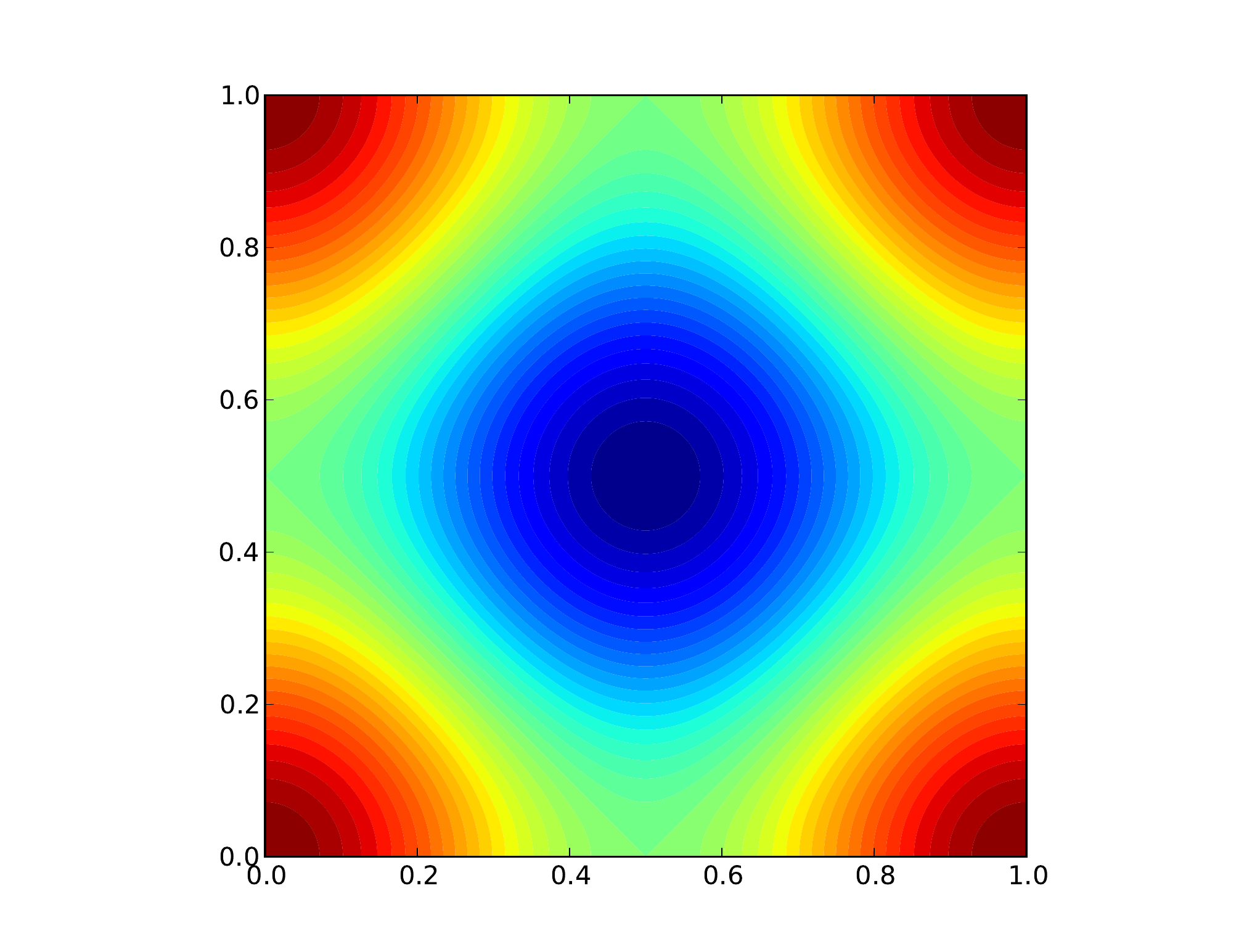} 
\end{center}
\caption{A few level sets of $a_{0}$}
\label{fig:contour}
\end{figure}

\begin{figure}[hp]
\begin{center}
\includegraphics[scale=0.4]{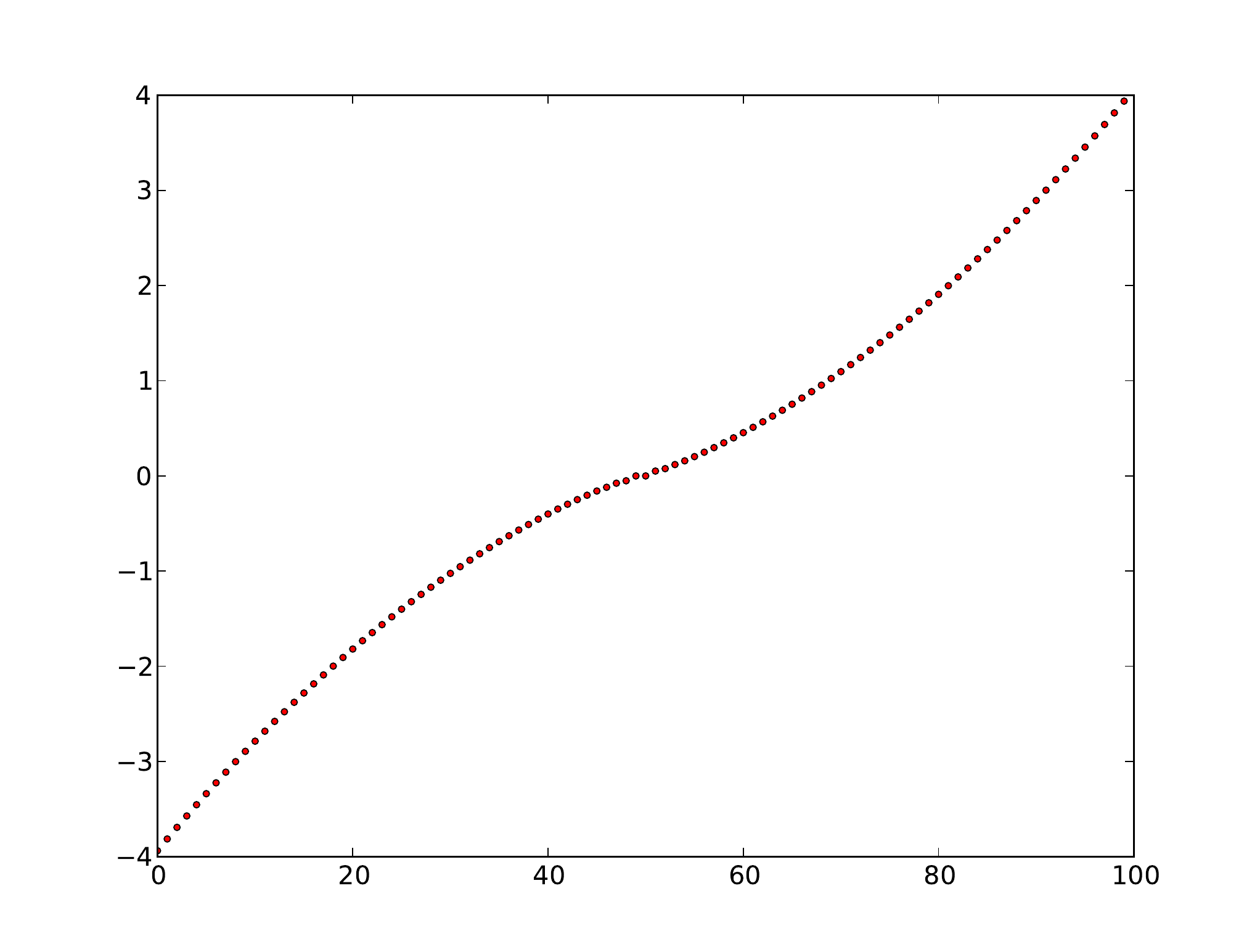} 
\end{center}
\caption{Ordered eigenvalues of $A_{k}$, $k=50$}
\label{fig:spectre}
\end{figure}

\begin{figure}[hp]
\subfigure[Modulus of the eigenfunction associated to the eigenvalue closest to $E$]{\includegraphics[scale=0.32]{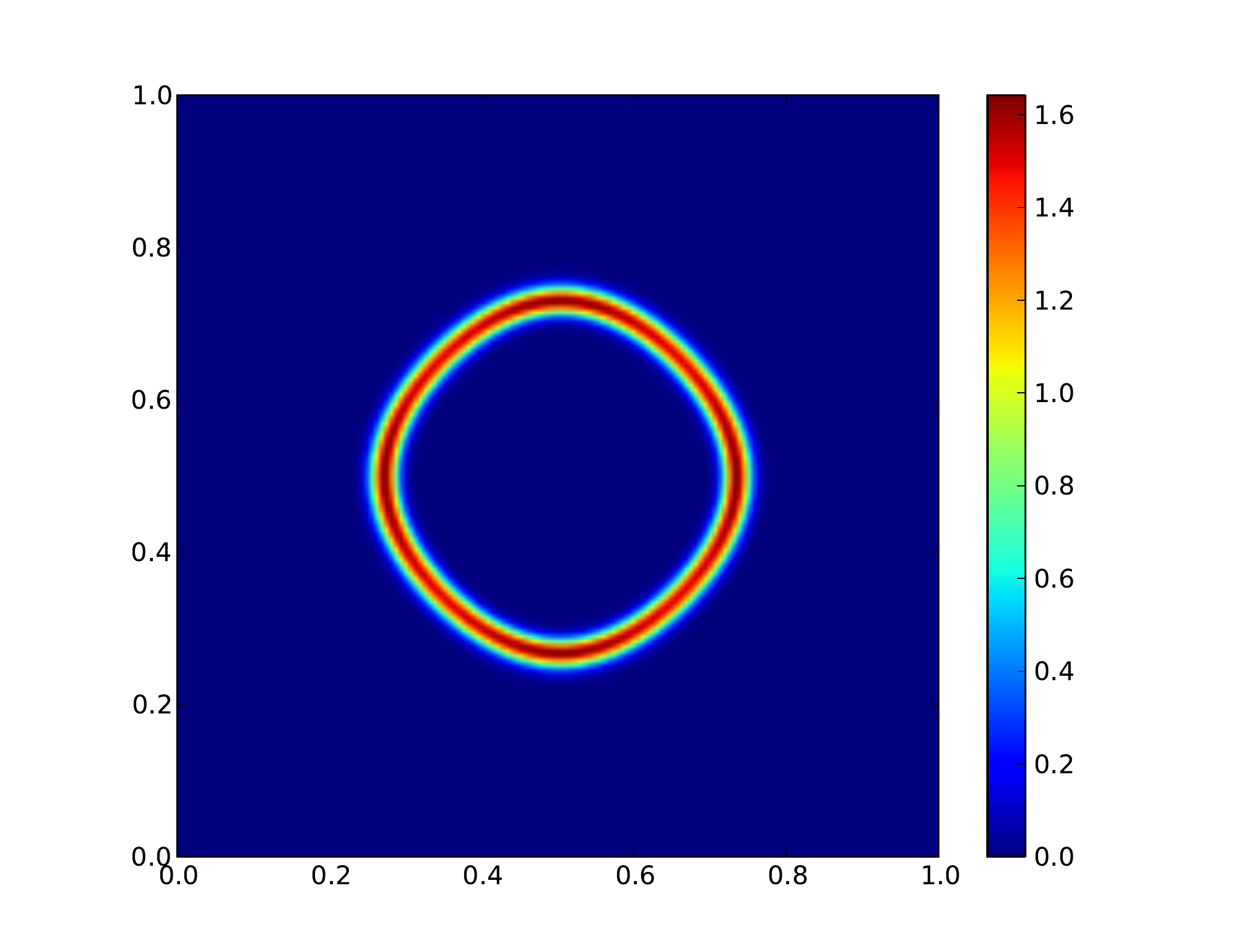} } 
\subfigure[Level set $a_{0}^{-1}(E)$]{\includegraphics[scale=0.32]{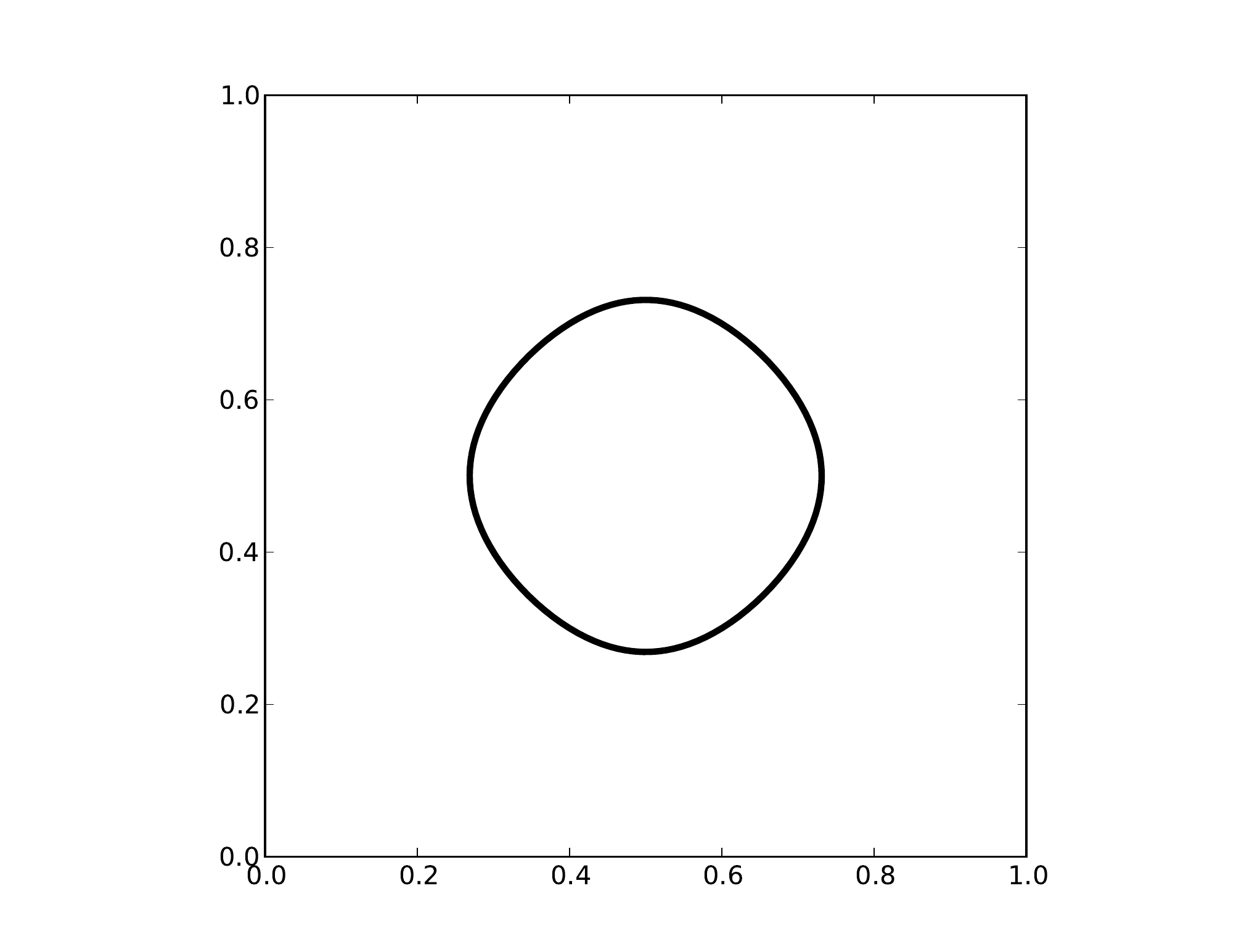} }

\caption{$E = a_{0}(0.7,0.6)$, $k=500$}
\label{fig:module}
\end{figure}

\begin{figure}[hp]
\subfigure[$k=50$]{\includegraphics[scale=0.32]{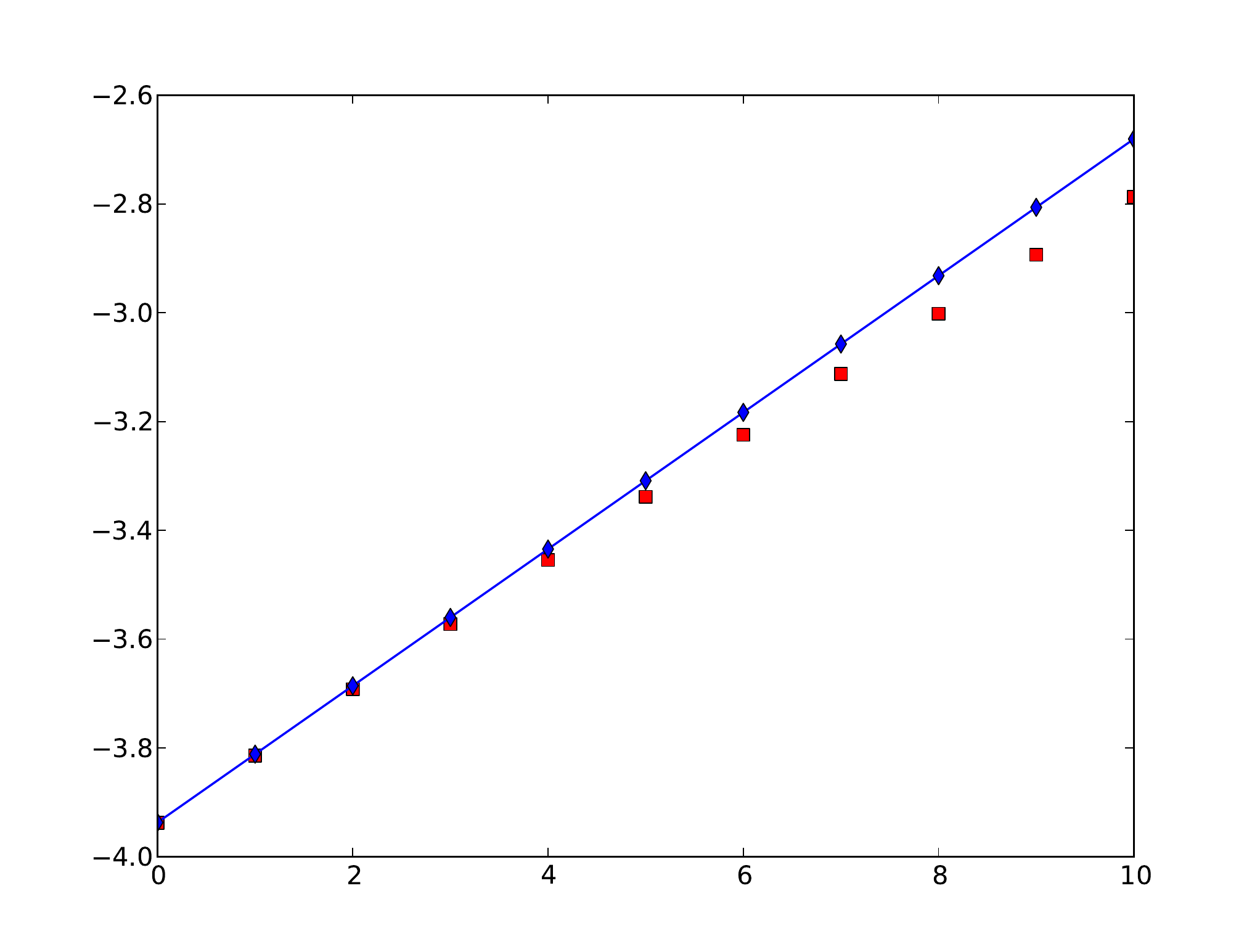} } 
\subfigure[$k=500$]{\includegraphics[scale=0.32]{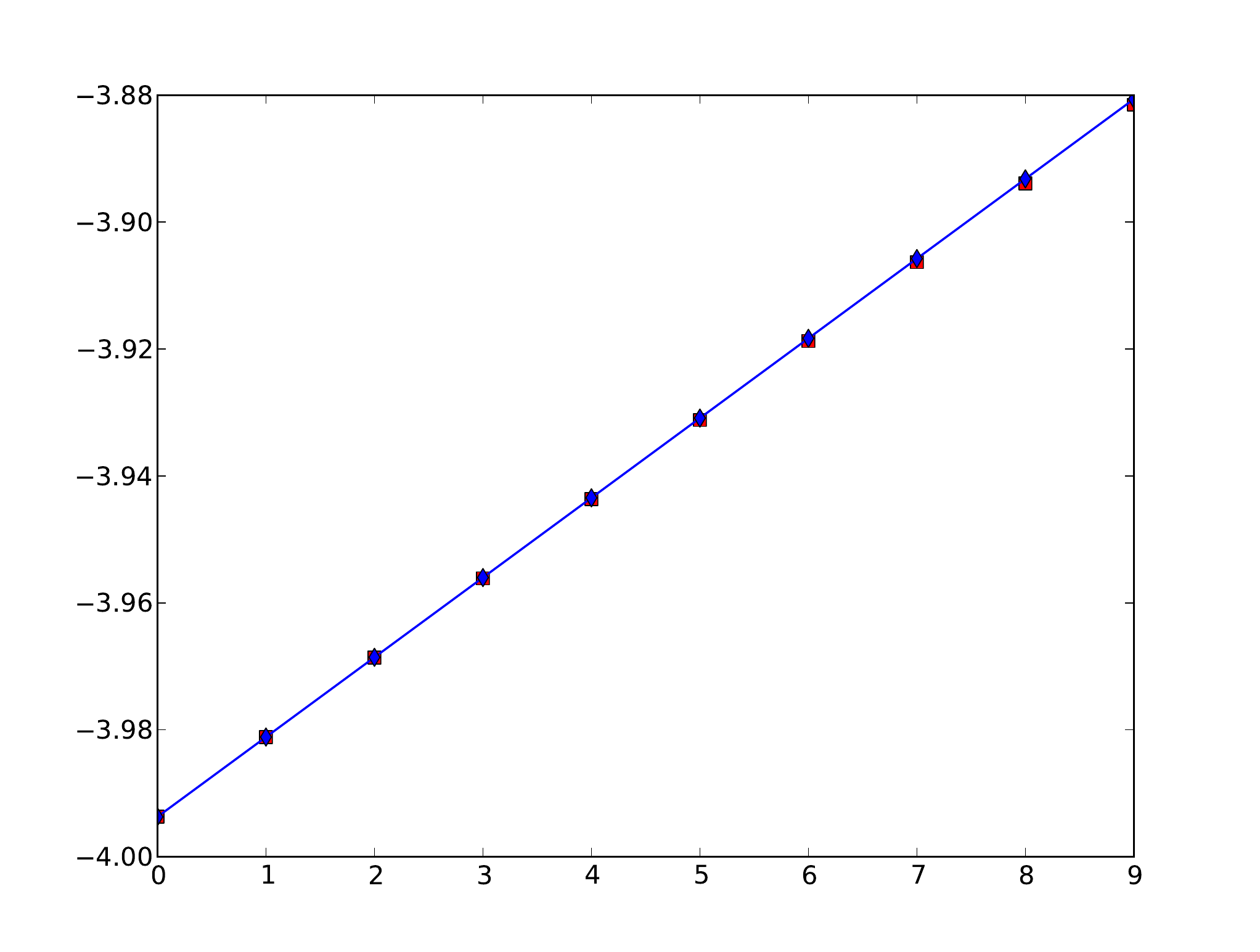} }
\caption{Eigenvalues in $[-4,-4 + 20\pi k^{-1}]$; in red squares, the eigenvalues of $A_{k}$ obtained numerically; in blue diamonds, the theoretical eigenvalues up to order $k^{-2}$ (\textit{i.e.} forgetting the $O(k^{-2})$ in formula (\ref{eq:DAtore}))}
\label{fig:zoom}
\end{figure}

\clearpage

\newpage

\appendix
\section{Appendix: proofs of the results of part \ref{subsection:ToepBarg}}

\begin{proof}[Proof of lemma \ref{lm:kernelToep}]
A first expression for this Schwartz kernel is
\begin{equation*} A_k(z_1,z_2) = \int_{\C} \Pi_k(z_1,z_3) a(z_3,k) \Pi_k(z_3,z_2) \ d\lambda(z_3) \end{equation*}
which is equal to
\begin{equation} \left( \frac{k}{2\pi} \right)^2 \int_{\C} \exp\left( -\frac{k}{2} \left( |z_1|^2 + |z_2|^2 + 2 |z_3|^2 - 2 z_1 \bar{z}_3 - 2 z_3 \bar{z}_2 \right)  \right) a(z_3,k) \ d\lambda(z_3). \label{eq:kernelholo}\end{equation}
This can be written
\begin{equation*} A_k(z_1,z_2) = \left( \frac{k}{2\pi} \right)^2 \int_{\C} \exp(ik\phi(z_1,z_2,z_3)) a(z_3,k) \ d\lambda(z_3) \end{equation*}
with the phase $\phi$ given by
\begin{equation*} \phi(z_1,z_2,z_3) = \frac{i}{2} \left( |z_1|^2 + |z_2|^2 + 2 |z_3|^2 - 2 z_1 \bar{z}_3 - 2 z_3 \bar{z}_2 \right). \end{equation*}
It is more convenient to use real variables: let $u_j = (x_j,\xi_j)$ be the point of $\R^2$ corresponding to $z_j = \frac{1}{\sqrt{2}} \left( x_j - i \xi_j \right) $. The phase $\phi$ reads
\begin{equation*} \phi(u_1,u_2,u_3) = \frac{i}{4} \left( \| u_1 - u_3 \|^2 + \| u_3 - u_2 \|^2 + 2i \omega_0(u_1 - u_2, u_3) \right). \end{equation*}
Using the identity
\begin{equation*}\| u_1 - u_3 \|^2 + \| u_3 - u_2 \|^2 = \frac{1}{2} \left( \| u_1 - u_2 \|^2 + \| 2 u_3 - u_1 - u_2 \|^2 \right),\end{equation*}
we can rewrite $\phi$ as
\begin{equation*} \phi(u_1,u_2,u_3) = \frac{i}{8} \| u_1 - u_2 \|^2 + \varphi(u_1,u_2,u_3) \end{equation*}
with
\begin{equation*} \varphi(u_1,u_2,u_3) = \frac{i}{4} \left( \frac{1}{2} \| 2 u_3 - u_1 - u_2 \|^2 + 2 i \omega_0(u_1 - u_2,u_3) \right) .\end{equation*} 
So we have $A_k(u_1,u_2) =  \exp\left( -\frac{k}{8}\|u_1 - u_2\|^2 \right) I_k(u_1,u_2)$ with
\begin{equation*} I_k(u_1,u_2) = \left( \frac{k}{2\pi} \right)^2 \int_{\R^2} \exp(ik\varphi(u_1,u_2,u_3)) a(u_3,k) \ d\lambda(u_3); \end{equation*}
a change of variable finally gives 
\begin{equation*} I_k(u_1,u_2) = \left( \frac{k}{2\pi} \right)^2 \int_{\R^2} \exp(ik\psi(u_1,u_2,u_3)) \ a\left(u_3 + \left(\frac{u_1 + u_2}{2}\right) ,k\right) \ d\lambda(u_3),\end{equation*}
with
\begin{equation*} \psi(u_1,u_2,u_3) = \frac{i}{2} \left( \| u_3 \|^2 + i \omega_0(u_1 - u_2,u_3) + i \omega_0(u_1,u_2) \right)  .\end{equation*}
To evaluate this integral, we cannot directly use the stationary phase lemma because $a(.,k)$ may not be compactly supported; we have to adapt it. In order to do so, we start by computing the critical locus
\begin{equation*} C_{\psi} = \left\{ (u_1,u_2,u_3) \in \C^3; d_{u_3}\psi(u_1,u_2,u_3) = 0 \ \text{and} \ \Im(\psi)(u_1,u_2,u_3) = 0 \right\} \end{equation*}
of $\psi$. It is clear that $\Im(\psi)(u_1,u_2,u_3) = 0$ if and only if $u_3 = 0$; moreover, the derivative of $\psi$ with respect to $u_3$ is given by
\begin{equation*} d_{u_3} \psi(u_1,u_2,u_3) = i \langle u_3,. \rangle - \frac{1}{2} \omega_0(u_1 - u_2,.), \end{equation*} 
hence $C_{\psi}$ is the set of $(u,u,0)$, $u \in \R^2$. Now, consider $\chi \in \classe{\infty}{(\R^2,\R^+)}$ equal to $1$ in the set $\left\{ \| x_3 \| \leq \delta \right\}$ for some $\delta > 0$ and with compact support, and decompose $I_k$ as $I_k = \left( \frac{k}{2\pi} \right)^2 (J_k + K_k)$, with 
\begin{equation*} J_k(u_1,u_2) =  \int_{\R^2} \exp(ik\psi(u_1,u_2,u_3)) \ a\left(u_3 + \left(\frac{u_1 + u_2}{2}\right) ,k\right) \chi(u_3) \ d\lambda(u_3) \end{equation*}
and $K_{k}(u_{1},u_{2})$ equal to the same integral replacing $\chi$ by $1-\chi$.
First, we show that $K_k$ is negligible. Choose $R > 0$ and consider the points $(u_1,u_2)$ belonging to the ball of $\R^4$ centered at the origin and of radius $R$. Writing the integral in $K_k$ in polar coordinates, we have
\begin{equation*} K_k(u_1,u_2) =  \int_{0}^{2\pi} \int_{\delta}^{+\infty} \exp\left( ik\Psi(\rho,\theta) \right) A(\rho,\theta,k)  \ d\rho \ d\theta \end{equation*}
where $\Psi(\rho,\theta) = \psi(u_1,u_2,(\rho \cos \theta,\rho \sin \theta))$ and 
\begin{equation*} A(\rho,\theta,k) =  \rho \ a\left((\rho \cos \theta, \rho \sin \theta) + \left(\frac{u_1 + u_2}{2}\right) ,k\right) (1 - \chi(\rho \cos \theta, \rho \sin \theta)). \end{equation*}
Performing successive integration by parts, it is easy to prove that for every positive integer $N$
\begin{equation*} K_k(u_1,u_2) = k^{-N} \int_{0}^{2\pi}  \int_{\delta}^{+\infty} \exp\left( ik\Psi(\rho,\theta) \right) D^N A(\rho,\theta,k) d\rho \ d\theta \end{equation*} 
where $D$ is the differential operator acting on $\classe{\infty}{([\delta,+\infty[ \times [0,2\pi] )}$ given by $Df = i \frac{\partial}{\partial \rho}\left( \left(\frac{\partial \Psi}{\partial \rho}\right)^{-1} f \right)$. Furthermore, using the facts that $\frac{\partial \Psi}{\partial \rho}(\rho,\theta) = i \rho - \frac{1}{2}(x_1-x_2) \sin \theta + \frac{1}{2} (y_1 - y_2) \cos \theta$, $D^N A$ is equal to a linear combination of terms of the form $\frac{\partial^{p} \Psi}{\partial \rho^{p}} \frac{\partial^{q} A}{\partial \rho^{q}}$ and $a(.,k)$ belongs to $\mathcal{S}_j^1$, we have the estimate 
\begin{equation*} \left| \int_{\delta}^{+\infty} \exp\left( ik\Psi(\rho,\theta) \right) D^N A(\rho,\theta,k) d\rho \right| \leq C_{N} \int_{\delta}^{+\infty} \exp\left(-\frac{k}{2}\rho^2\right) w(\rho,\theta) \ d\rho \end{equation*}
with
\begin{equation*} w(\rho,\theta) =  \rho^{j_1} \left( 1 + \left| \rho(\cos \theta, \sin \theta) + \frac{u_1 + u_2}{2} \right|^2 \right)^{j_2} \end{equation*}
for some $C_{N} > 0$ and $j_1,j_2 \in \Z$. If $j_2 < 0$, this last integral can be bounded by $\int_{\delta}^{+\infty} \rho^{j_1} \exp\left(-\frac{k}{2}\rho^2\right) d\rho = O(k^{-1/2})$; if $j_2 > 0$, we have
\begin{equation*}\left( 1 + \left| \rho(\cos \theta,\sin \theta) + \frac{u_1 + u_2}{2} \right|^2 \right)^{j_2} \leq \left( 1 + |\rho|^2 + |R|^2 \right)^{j_2} \end{equation*}
and hence
\begin{equation*} \int_{\delta}^{+\infty} \exp\left(-\frac{k}{2}\rho^2\right) w(\rho,\theta) \ d\rho \leq \tilde{C}_N \int_{\delta}^{+\infty} \exp\left(-\frac{k}{2}\rho^2\right) \rho^{j_3} \ d\rho = O(k^{-1/2}).\end{equation*}
This shows that $K_k \leq c_N k^{-N}$ on the ball of radius $R$ for some $c_N > 0$. We treat the derivatives of $K_k$ in the same way.

It remains to estimate $J_k$. Since the second derivative of $\psi$ with respect to $u_3$ is equal to $i \text{Id}$, and since
\begin{equation*}  \begin{split} \psi(u_1,u_2,u_3) & =  \frac{i}{8}  \| u_1 - u_2 \|^2  - \frac{1}{2} \omega_0(u_1,u_2) \\ & + \frac{1}{4} \left( i(\xi_1 - \xi_2) + 2 x_3 \right) \partial_{x_4}\psi(u_1,u_2,u_3) \\ &
 + \frac{1}{4} \left( i(x_2 - x_1) + 2 \xi_3 \right) \partial_{\xi_4}\psi(u_1,u_2,u_3), \end{split} \end{equation*}
the stationary phase lemma \cite[section $7.7$]{Hor} gives 
\begin{equation*} J_k(u_1,u_2) = \frac{2 \pi}{k} \exp\left(\frac{i}{8} \| u_1 - u_2 \|^2 - \frac{1}{2} \omega_{0}(u_{1},u_{2}) \right) \tilde{a}(u,u,k) + O(k^{-\infty}), \end{equation*}
where $\tilde{a}(.,.,k)$ belongs to $\mathcal{S}_j^2$; the coefficients $\tilde{a}_{\ell}(u_1,u_2,k)$ of its asymptotic expansion, which we do not write here, are linear combinations of derivatives of the $a_{m}, m \geq 0$, evaluated at $\frac{u_1 + u_2}{2}$. However, the values of $a(.,.,k)$ along the diagonal of $\C^2$ can be easily computed, because a number of terms vanish: for $z$ in $\C$, we have
\begin{equation*} \tilde{a}(z,z,k) = \left(\sum_{\ell \geq 0} \frac{k^{-\ell}}{\ell !} \Delta^{\ell} a \right)(z,k) = \left( \exp{k^{-1}\Delta} a \right)(z,k). \end{equation*}
Putting this together with the fact that $K_k$ is negligible, we obtain the result. 
\end{proof} 
\begin{proof}[Proof of lemma \ref{lm:covBarg}]
Formula (\ref{eq:kernelholo}) shows that the kernel $A_k$ is a holomorphic section of $L_0^k \boxtimes L_0^{-k}$; differentiating equation (\ref{eq:SchToep}), this implies that the sequences of functions $z \in \C \mapsto \frac{\partial \tilde{a}}{\partial \bar{z}_1}(z,z,k)$ and $z \in \C \mapsto \frac{\partial \tilde{a}}{\partial z_2}(z,z,k)$ are negligible. Hence, we have for $\ell \geq 0$ and $z \in \C$
\begin{equation*} \frac{\partial \tilde{a}_{\ell}}{\partial \bar{z}_1}(z,z) = 0 = \frac{\partial \tilde{a}_{\ell}}{\partial z_2}(z,z).\end{equation*} 
Thanks to these holomorphy conditions, the fact that $\tilde{a_{\ell}}$ vanishes on the diagonal implies that it vanishes to all order along the diagonal. We can easily adapt lemma 1 of \cite{Cha1} to show that this yields the negligibility of $\exp\left(-\frac{k}{4} \left| z_{1} - z_{2} \right|^2 \right) \tilde{a}(z_{1},z_{2},k)$. Injecting this in formula (\ref{eq:SchToep}) gives the result.
\end{proof}
\begin{proof}[Proof of corollary \ref{cor:compoToep}]
The kernel of $C_k$ reads
\begin{equation*} C_k(z_1,z_2) = \int_{C} A_k(z_1,z_3) B_k(z_3,z_2) \ d\lambda(z_3). \end{equation*}
Using the representations of $A_k$ and $B_k$ given by lemma \ref{lm:kernelToep}, this yields
\begin{equation*} \begin{split} C_k(z_1,z_2) = \left( \frac{k}{2\pi} \right)^2 \int_{\C} \exp\left( i k\phi(z_1,z_2,z_3)  \right) \tilde{a}(z_1,z_3,k) \tilde{b}(z_3,z_2,k) \ d\lambda(z_3) \\ + R_k \exp\left( -Ck |z_1 - z_2|^2\right), \end{split} \end{equation*}
with $C > 0$, $R_k$ negligible and
\begin{equation*} \phi(z_1,z_2,z_3) = \frac{i}{2} \left( |z_1|^2 + |z_2|^2 + 2 |z_3|^2 - 2 z_1 \bar{z}_3 - 2 z_3 \bar{z}_2 \right). \end{equation*}
Using the same technique as in the proof of lemma \ref{lm:kernelToep}, we show that 
\begin{equation*} \begin{split} C_k(z_1,z_2) = \frac{k}{2\pi} \exp\left( -\frac{k}{2} \left( |z_1|^2 + |z_2|^2 - 2 z_1 \bar{z}_2 \right) \right) \tilde{c}(z_1,z_2,k)  \\ + R'_k \exp\left(-C' k|z_1 - z_2|^2\right) \end{split} \end{equation*}
with $C' > 0$, $R'_k$ negligible, and $\tilde{c}(.,.,k) \in \mathcal{S}_{j + j'}^2$. Now, consider the function $\check{c}(.,k)$ defined by $\check{c}(z,k) = \tilde{c}(z,z,k)$ for $z$ in $\C$, and put $c(.,k) = \left( \exp\left(-{k^{-1}\Delta}\right) \check{c}\right)(.,k)$. Then $c(.,k)$ belongs to $\mathcal{S}_{j + j'}^1$ and, by lemma \ref{lm:kernelToep}, the Toeplitz operator $D_k = \text{Op}(c(.,k))$ admits a Schwartz kernel of the form 
\begin{equation*} \begin{split} D_k(z_1,z_2) = \frac{k}{2 \pi} \exp\left( -\frac{k}{2} \left( |z_1|^2 + |z_2|^2 - 2 z_1 \bar{z}_2 \right) \right) \tilde{d}(z_1,z_2,k)  \\ + R''_k \exp\left(-C''k|z_1 - z_2|^2\right), \end{split} \end{equation*}
where $\tilde{d}(.,.,k)$ belongs to $\mathcal{S}_{j}^2$, $R''_k$ is negligible, $C''$ is a positive constant and for every $z$ in $\C$, $\tilde{d}(z,z,k) = \tilde{c}(z,z,k)$. Lemma \ref{lm:covBarg} yields that $C_k = D_k + R'''_k \exp\left(-C'''k|z_1 - z_2|^2\right)$ for some $C''' > 0$ and $R'''_k$ negligible.

It remains to compute the contravariant symbol of $C_k$. For $z$ in $\C$, put $\breve{a}(z,w,k) = \tilde{a}(z,z + w,k)$ and $\breve{b}(z,w,k) = \tilde{b}(z + w,z,k)$. One has 
\begin{equation*} \check{c} = \left( \exp \left(k^{-1}\Delta_w \right) \breve{a} \breve{b} \right)_{|w=0} \end{equation*}
with $\Delta_w$ the holomorphic laplacian with respect to $w$; using lemma \ref{lm:covBarg}, we find  
\begin{equation*} \check{c}(z,k) = \left( \exp \left(k^{-1} \frac{\partial}{\partial \bar{u}} \frac{\partial}{\partial v} \right) \check{a}(u,k)) \check{b}(v,k) \right)_{|u = v = z} \end{equation*}
up to a negligible term. 
Now, since $c(.,k) = \left( \exp\left(-{k^{-1}\Delta}\right) \check{c}\right)(.,k)$, this yields
\begin{equation*} c(z,k) = \left( \exp\left( -k^{-1} \frac{\partial}{\partial u} \frac{\partial}{\partial \bar{v}}\right) a(u,k) b(v,k) \right)_{|u=v=z} \end{equation*}
up to a negligible term, which was to be proved.
\end{proof}

\section*{Acknowledgements}
I would like to thank San V\~u Ng\d{o}c and Laurent Charles for their precious help and numerous readings of the manuscript.

\bibliographystyle{plain}
\bibliography{BSell}

\end{document}